\documentclass[leqno,onefignum,onetabnum]{siamltex1213}
\usepackage{amsfonts,amsmath,amssymb}
\usepackage{booktabs,ctable,threeparttable,multirow}
\usepackage{graphicx,boxedminipage,appendix}
\usepackage{bm,enumerate,algorithm,algorithmic}
\usepackage[caption=false]{subfig}
\usepackage{mathrsfs}

\usepackage{amsthm}
\usepackage{lineno}
\usepackage{enumerate}
\usepackage{epstopdf}

\newtheorem{theoremmy}{Theorem}[section]

\newtheorem{exper}[theoremmy]{Experiment}
\newtheorem{propositionmy}[theoremmy]{Proposition}
\numberwithin{equation}{section}

\newcommand{\sep}[0]{{\rm sep}}
\newcommand{\N}[0]{\mathcal{N}}
\newcommand{\X}[0]{\mathcal{X}}
\newcommand{\UU}[0]{\mathcal{U}}
\newcommand{\V}[0]{\mathcal{V}}
\newcommand{\OO}[0]{\mathcal{O}}
\newcommand{\bsmallmatrix}[1]{\begin{bmatrix}\begin{smallmatrix}#1\end{smallmatrix}\end{bmatrix}}

\title{A cross-product free Jacobi-Davidson type method for computing
a partial generalized singular value decomposition of a large matrix
pair\thanks{Supported by the National Natural Science
Foundation of China (No.11771249).}}

\author{
Jinzhi Huang\thanks{Department of Mathematical Sciences,
Tsinghua University, 100084 Beijing, China
(\url{huangjz15@mails.tsinghua.edu.cn}).}
\and
Zhongxiao Jia\thanks{Corresponding author. Department of Mathematical Sciences,
Tsinghua University, 100084 Beijing, China
(\url{jiazx@tsinghua.edu.cn})}
}

\begin{document}
\maketitle

\begin{abstract}
A Cross-Product Free (CPF) Jacobi-Davidson (JD) type method is
proposed  to compute a partial generalized singular value decomposition
(GSVD) of a large regular matrix pair $(A,B)$. It
implicitly solves the mathematically equivalent generalized eigenvalue
problem of $(A^TA,B^TB)$ but does not explicitly
form the cross-product matrices and thus avoids the possible accuracy loss
of the computed generalized singular values and generalized singular vectors.
The method is an inner-outer iteration method, where the expansion
of the right searching subspace forms the inner iterations that approximately
solve the correction equations involved and
the outer iterations extract approximate GSVD components with respect to the
subspaces.
Some convergence results are established for the inner and outer iterations,
based on some of which practical stopping criteria are designed for the inner
iterations. A thick-restart CPF-JDGSVD algorithm with deflation is developed to
compute several GSVD components.
Numerical experiments illustrate the efficiency of the algorithm.
\end{abstract}

\begin{keywords}
Generalized singular value decomposition,
generalized singular value, generalized singular vector,
extraction approach, subspace expansion,
Jacobi-Davidson method, correction equation,
inner iteration, outer iteration
\end{keywords}

\begin{AMS}
65F15, 15A18, 15A12, 65F10
\end{AMS}

\pagestyle{myheadings}
\thispagestyle{plain}
\markboth{JINZHI HUANG AND ZHONGXIAO JIA}{A JACOBI-DAVIDSON TYPE METHOD FOR THE GSVD COMPUTATION}

\section{Introduction}\label{sec:1}
The generalized singular value decomposition (GSVD) of a matrix
pair is first introduced by Van Loan \cite{van1976generalizing}
and then developed by Paige and Saunders \cite{paige1981towards}.
It has become an important analysis means and
computational tool \cite{golub2012matrix},
and has been used extensively in, e.g.,
solutions of discrete linear ill-posed problems
\cite{hansen1998rank}, weighted or generalized least squares problems
\cite{bjorck1996numerical},
information retrieval \cite{howland2003structure},
linear discriminant analysis \cite{park2005relationship},
and many others \cite{betcke2008generalized,chu1987singular,
golub2012matrix,kaagstrom1984generalized,vanhuffel}.

Let $A\in\mathbb{R}^{m\times n}$ and $B\in\mathbb{R}^{p\times n}$
with $m\geq n$ be large matrices, and assume that the
stacked matrix $\bsmallmatrix{A\\B}$ has full column rank, i.e.,
$\N(A)\cap\N(B)=\{\mathbf{0}\}$ with $\N(A)$ and $\N(B)$ the null
spaces of $A$ and $B$, respectively.
Then such matrix pair $(A,B)$ is called {\em regular}.
Denote $q_1=\dim(\N(A))$ and $q_2=\dim(\N(B))$,  $q=n-q_1-q_2$,
and $l_1=\dim(\N(A^T))$ and $l_2=\dim(\N(B^T))$.
Then the GSVD of the regular matrix pair $(A,B)$ is
\begin{equation}\label{gsvd}
\left\{\begin{aligned}
&U^TAX=\Sigma_A=\diag\{C,\mathbf{0}_{l_1, q_1},I_{q_2}\}, \\
&V^TBX=\Sigma_B=\diag\{S,I_{q_1},\mathbf{0}_{l_2, q_2}\},
\end{aligned}\right.
\end{equation}
where $X=[X_q,X_{q_1},X_{q_2}]$ is nonsingular,
$U=[U_q,U_{l_1},U_{q_2}]$ and $V=[V_q,V_{q_1},V_{l_2}]$  are orthogonal,
and $C=\diag\{\alpha_1,\dots,\alpha_q\}$ and
$S=\diag\{\beta_1,\dots,\beta_q\}$ are diagonal matrices that satisfy
\begin{equation*}
  0<\alpha_i,\beta_i<1 \quad\mbox{and}\quad
 \alpha_i^2+\beta_i^2=1,  \quad i=1,\dots,q;
\end{equation*}
see \cite{paige1981towards}.
Here, in order to distinguish the  block submatrices of $X,U,V$,
we have adopted the subscripts to
denote their column and row numbers, and have denoted by $I_{k}$ and
$\mathbf{0}_{k,l}$ the identity matrix of order $k$ and zero matrix
of order $k\times l$, respectively.
The subscripts of identity and zero matrices will be omitted in the
sequel when their orders are clear from the context. It follows
from \eqref{gsvd} that $X^T(A^TA+B^TB)X=I_n$, i.e., the columns of $X$
are $(A^TA+B^TB)$-orthonormal.

The GSVD components $(\mathbf{0}_{l_1, q_1},I_{q_1},U_{l_1},V_{q_1},X_{q_1})$
and $(I_{q_2},\mathbf{0}_{l_2, q_2},U_{q_2},V_{l_2},X_{q_2})$
are associated with the zero and infinite generalized singular
values of $(A,B)$, called the {\em trivial} ones, and the columns of
$U_{l_1}$, $V_{l_2}$ and
$X_{q_1}$, $X_{q_2}$ form orthonormal and $(A^TA+B^TB)$-orthonormal
bases of $\N(A^T)$, $\N(B^T)$ and $\N(A)$, $\N(B)$, respectively.
Denote by $u_i,v_i$ and $x_i$ the $i$-th columns of
$U_q$, $V_q$ and $X_q$, respectively, $i=1,\dots,q.$
The quintuple $(\alpha_i,\beta_i,u_i,v_i,x_i)$ is called a {\em nontrival} GSVD
component of $(A,B)$ with the generalized singular value
$\sigma_{i}=\frac{\alpha_i}{\beta_i}$,
the left generalized singular vectors $u_i$, $v_i$ and the right
generalized singular vector $x_i$.
We also refer to a pair $(\alpha_i,\beta_i)$ as a
generalized singular value of $(A,B)$.

For a given target $\tau>0$, assume that the nontrivial generalized
singular values $\sigma_i,\,i=1,2,\ldots,q$ of $(A,B)$ are labeled as
\begin{equation}\label{gsvalue}
  |\sigma_1-\tau|\leq|\sigma_2-
  \tau|\leq\dots\leq |\sigma_{\ell}-\tau|<|\sigma_{\ell+1}-\tau|\leq\dots\leq|\sigma_q-\tau|.
\end{equation}
We are interested in computing the $\ell$ GSVD components
$(\alpha_i,\beta_i,u_i,v_i,x_i)$ corresponding to the
generalized singular values closest
to $\tau$. If $\tau$ is inside the spectrum of the nontrivial
generalized singular values of $(A,B)$,
then $(\alpha_i,\beta_i,u_i,v_i,x_i)$, $i=1,\dots,\ell$, are called
interior GSVD components of $(A,B)$; otherwise, they are called
the extreme, i.e., largest or smallest, ones.
A large number of GSVD components, i.e., $\ell\gg 1$,
some of which are interior ones, may be required in applications,
including nonlinear dimensionality reduction and data
science \cite{chuiwang2010,coifman2006,coifman2005}.
Without loss of generality, we always assume that
$\tau$ is not equal to any generalized singular value
of $(A,B)$.

Zha \cite{zha1996computing} proposes a joint bidiagonalization (JBD)
method for computing extreme GSVD components of the large
matrix pair $(A,B)$.
At each step of JBD, one needs to
solve an $(m+p)\times n$ least squares problem with
the coefficient matrix $\bsmallmatrix{A\\B}$,
which may be costly using an iterative solver \cite{jia2018joint}.
Hochstenbach \cite{hochstenbach2009jacobi} presents a Jacobi--Davidson
(JD) type GSVD (JDGSVD) method to compute $\ell$
GSVD components of $(A,B)$ with the full column rank $B$,
where, at each expansion step,
a correction equation of dimension $m+n$ needs to be solved iteratively with
low or modest accuracy; see \cite{huang2019inner,
jia2014inner,jia2015harmonic}.
The upper $m$-dimensional part and the lower $n$-dimensional part of
the approximate solution are used to expand
one of the left searching subspaces and the right searching subspace.
The JDGSVD method formulates the GSVD of
$(A,B)$ as the mathematically equivalent generalized eigendecomposition
of the augmented matrix pair
$\left(\bsmallmatrix{&A\\A^T&},\bsmallmatrix{I&\\&B^TB}\right)$
for the full column rank $B$ (resp.
$\left(\bsmallmatrix{&B\\B^T&},\bsmallmatrix{I&\\&A^TA}\right)$
for the full column rank $A$),
computes the relevant eigenpairs and recovers the approximate
GSVD components from the converged eigenpairs.

The side effects of involving the cross-product matrix $B^TB$ in
JDGSVD method are twofold.
First, in the extraction phase, one is required to compute
a $B^TB$-orthonormal basis of the right searching subspace,
which is numerically unstable when $B$ is ill conditioned.
Second, it is shown in \cite{huang2020choices} that
the error of the computed eigenvector is bounded by the size of
the perturbations times a multiple $\kappa(B^TB)=\kappa^2(B)$, where
$\kappa(B)=\sigma_{\max}(B)/\sigma_{\min}(B)$ denotes the $2$-norm
condition number of $B$ with $\sigma_{\max}(B)$ and $\sigma_{\min}(B)$
being the largest and smallest singular values of $B$, respectively.
Consequently, with an ill-conditioned $B$,
the computed GSVD components may have very poor accuracy and
has been numerically confirmed \cite{huang2020choices}.
We remark that all the eigenvalue-based type GSVD methods share this
shortcoming of JDGSVD. The results in \cite{huang2020choices} show
that if $B$ is ill conditioned but $A$ has full column rank
and is well conditioned then the JDGSVD method can be applied to the
matrix pair $(B,A)$ and compute the corresponding approximate GSVD
component with high accuracy \cite{huang2020choices}.
However, a reliable estimation of the condition numbers of $A$ and $B$
is numerically challenging and may be costly. As a result,
it is difficult to choose a proper formulation in applications.

Zwaan and Hochstenbach \cite{zwaan2017generalized} present
a generalized Davidson (GDGSVD) method and a multidirectional
(MDGSVD) method, which are designed to compute an {\em extreme} partial GSVD of $(A,B)$. The methods avoid any cross product or matrix-matrix product
by applying the standard extraction approach to $(A,B)$ directly
for computing approximate GSVD components with respect to the given left and right
searching subspaces with the two left subspaces formed by premultiplying
the right subspace with $A$ and $B$, respectively. At each iteration
of the GDGSVD method,
the right searching subspace is spanned by the residuals of the generalized
Davidson method \cite[Sec. 11.2.4 and Sec. 11.3.6]{bai2000} applied to
the generalized eigenvalue problem of $(A^TA,B^TB)$; in the MDGSVD method,
a truncation technique is designed
to discard an inferior search direction so as to improve the searching subspaces.
Exploiting the Kronecker canonical form of a regular matrix pair \cite{stewart90},
Zwaan \cite{zwaan2019} shows that the GSVD problem of $(A,B)$ can be formulated
as a certain $(2m+p+n)\times (2m+p+n)$ generalized eigenvalue
problem without using any cross product or any other matrix-matrix product.
Currently, such formulation is of major theoretical value because
the nontrivial eigenvalues and eigenvectors of the structured generalized eigenvalue problem
are always complex with the generalized eigenvalues being the conjugate quaternions
$(\sqrt{\sigma_j},-\sqrt{\sigma_j},\mathrm{i}\sqrt{\sigma_j},-\mathrm{i}
\sqrt{\sigma_j})$ with $\mathrm{i}$ the imaginary unit and
the associated right generalized eigenvectors being
\begin{eqnarray*}
&&[u_j^T,x_j^T/\beta_j,\sqrt{\sigma_j}u_j^T,\sqrt{\sigma_j}v_j^T]^T,\
[-u_j^T,-x_j^T/\beta_j,\sqrt{\sigma_j}u_j^T,\sqrt{\sigma_j}v_j^T]^T,
\\
&&[-\mathrm{i}u_j^T,\mathrm{i}x_j^T/\beta_j,\sqrt{\sigma_j}u_j^T,-\sqrt{\sigma_j}v_j^T]^T,\
 [\mathrm{i}u_j^T,\mathrm{i}x_j^T/\beta_j,-\sqrt{\sigma_j}u_j^T,-\sqrt{\sigma_j}v_j^T]^T.
\end{eqnarray*}
As is clear, the size of the generalized eigenvalue problem is
much bigger than that of the GSVD of $(A,B)$. It is also
unclear what the conditioning
of eigenvalues and eigenvectors of this problem is.
Furthermore, there has been no structure-preserving algorithm for
the complicated structured generalized eigenvalue problem
in \cite{zwaan2019}. It
will be extremely difficult and highly challenging to
seek for a numerically stable structure-preserving efficient algorithm for
that structured generalized eigenvalue problem.


In order to compute GSVD components accurately,
it is appealing to propose and develop algorithms that
work on $A$ and $B$ directly.
In this paper, we first propose a basic Cross-Product Free (CPF)
JD type method for computing one, i.e., $\ell=1$, GSVD component
of $(A,B)$, which is referred to as
CPF-JDGSVD in the sequel. As done in the GDGSVD and MDGSVD
methods \cite{zwaan2017generalized},
instead of constructing left and right searching subspaces
separately or independently, given a {\em right} searching subspace,
the CPF-JDGSVD method generates the corresponding two left searching
subspaces by acting $A$ and $B$ on the right subspace, respectively,
and constructs their orthonormal bases by computing two thin QR
factorizations of the matrices that are formed by premultiplying
the matrix consisting of the orthonormal basis vectors
of the right subspace with $A$
and $B$, respectively. But unlike \cite{zwaan2017generalized},
at the extraction stage, our method projects the GSVD of $(A,B)$
onto the left and right searching subspaces without involving
$A^TA$ and $B^TB$, and obtains an
approximation to the desired GSVD component of $(A,B)$ by computing
the GSVD of the small sized projection matrix pair.
To be practical, we develop a thick-restart CPF-JDGSVD algorithm with
deflation for computing several, i.e., $\ell>1$, GSVD components.

We shall, for the first time, give a theoretical justification that
the left searching subspaces are
as good as the right one as long as the desired generalized
singular value $\sigma$ is not very small or large, that is,
the distances of the desired left generalized singular vectors and
the left subspaces
are as small as that of the desired right generalized
singular vector and the right subspace.
We give a detailed derivation of certain new $n\times n$ correction
equations involved in CPF-JDGSVD, whose solutions are exploited to
expand right searching subspaces.
The correction equations are supposed to be approximately solved iteratively,
called inner iterations, and CPF-JDGSVD is an inner-outer iterative method
with extraction steps of approximate
GSVD components called outer iterations.
We establish a convergence result on the approximate generalized
singular values in terms of the residual norms.
Meanwhile, we derive some results on the inner iterations in CPF-JDGSVD
and obtain the asymptotic condition numbers of the correction equations.
Based on them and analysis,
we propose practical stopping criteria for the inner
iterations, making the computational cost of the inner iterations minimal
at each outer iteration and guarantee that CPF-JDGSVD behaves
like the exact CPF-JDGSVD where the correction equations are solved exactly.

The rest of this paper is organized as follows.
In Section~\ref{sec2}, we propose the CPF-JDGSVD method and present
some theoretical results on its rationale and convergence.
In Section~\ref{sec:4}, we derive correction equations and establish some
properties of them and some results on the inner iterations.
Based on these results, we design practical stopping criteria for the inner
iterations in CPF-JDGSVD.
In Section~\ref{sec:5}, we propose a CPF-JDGSVD algorithm with
restart and deflation for computing more than one GSVD components.
Numerical experiments are reported in Section~\ref{sec:7} to
demonstrate the performance of CPF-JDGSVD and
show its superiority to JDGSVD in \cite{hochstenbach2009jacobi}
in terms of efficiency and accuracy. We should remind that,
among the available algorithms, JDGSVD is the only available one
that can be used to make a comparison with CPF-JDGSVD
for computing the generalized
singular values closest to a given target $\tau$.
Finally, we conclude this paper in Section~\ref{sec:8}.

Throughout the paper, denote by $\|\cdot\|$ the $2$-norm of
a vector or matrix, and assume that $\|A\|$ and $\|B\|$ themselves are modest,
which can be achieved by suitable scaling.
Since the stacked matrix $\bsmallmatrix{A\\B}$ is better and can be much better
conditioned than both $A$ and $B$,
we will assume that $\bsmallmatrix{A\\B}$ is well conditioned,
which is definitely true, provided one of $A$ and $B$ is well conditioned.

\section{The basic CPF-JDGSVD algorithm}\label{sec2}

We propose a basic CPF-JDGSVD method for computing one
GSVD component $(\alpha,\beta,u,v,x):=(\alpha_1,\beta_1,u_1, v_1, x_1)$
of $(A,B)$
corresponding to the generalized singular value $\sigma:=\sigma_1$
closest to the target $\tau$.
The method includes three major
ingredients: (i) the construction of left searching subspaces for a given
right searching subspace,
(ii) an extraction approach of approximate GSVD components, and (iii) an
expansion approach of the right subspace.
We will prove that the accuracy of the left searching subspaces constructed
is similar to that of the right subspace,
and will establish an important convergence result on
the approximate generalized singular values.

\subsection{Extraction approach}\label{sec:2}

At iteration $k$, assume that a $k$-dimensional right searching
subspace $\mathcal{X}$ is available, from which we seek an
approximation to the right generalized singular vector $x$.
For the left generalized singular vectors $u$ and $v$, since
$Ax=\alpha u$ and $Bx=\beta v$, it is natural to construct
$\mathcal{U}:=A\mathcal{X}$ and $\mathcal{V}:=B\mathcal{X}$
as left searching subspaces and seek approximations
to $u$ and $v$ from them, respectively.

We shall present a theoretical justification that such $\mathcal{U}$ and
$\mathcal{V}$ contain the same accurate information on the desired $u$ and $v$
as $\mathcal{X}$ does on $x$.
As a result, it is possible to extract approximate left and right
generalized singular vectors with similar accuracy.

\begin{propositionmy}\label{thm:1}
Let $\X$ be a given right searching subspace
and $\mathcal{U}=A\X$ and $\mathcal{V}=B\X$ be
the left searching subspaces. Then for the desired
right and left generalized singular vectors $x$ and
$u$, $v$ of $(A,B)$ associated with the generalized singular
value $\sigma=\frac{\alpha}{\beta}$ it holds that
  \begin{eqnarray}
    \sin\angle(\UU,u)&\leq&
    \frac{\|A\|\|x\|}{\alpha}\sin\angle(\X,x),\label{sinUu} \\
    \sin\angle(\V,v)&\leq&
    \frac{\|B\|\|x\|}{\beta}\sin\angle(\X,x). \label{sinVv}
  \end{eqnarray}
\end{propositionmy}

\begin{proof}
For an arbitrary vector $x^{\prime}\in\X$,
by the definition of sine of the angle between arbitrary two nonzero vectors,
we have
\begin{eqnarray}
   \sin\angle(Ax^{\prime},Ax)
   &=&\min_{\mu}\frac{\|Ax-\mu Ax^{\prime}\|}{\|Ax\|}
   =\min_{\mu}\frac{\|A(x-\mu x^{\prime})\|}{\|Ax\|} \nonumber \\
   &\leq& \frac{\|A\|\|x\|}{\|Ax\|}\min_{\mu}
   \frac{\|x-\mu x^{\prime}\|}{\|x\|}  \nonumber\\
   &=&\frac{\|A\|\|x\|}{\alpha}\sin\angle(x^{\prime},x), \label{sinwithA}
\end{eqnarray}
where the last relation holds since $Ax=\alpha u$
with $\|u\|=1$. Therefore, we obtain
\begin{eqnarray*}
\sin\angle(\UU,u)&=&\sin\angle(A\X,Ax)
=\min_{x^{\prime}\in\X}\sin\angle(Ax^{\prime},Ax)\\
&\leq&\frac{\|A\|\|x\|}{\alpha}\min_{x^{\prime}\in\X}
\sin\angle(x^{\prime},x)\\
&=&\frac{\|A\|\|x\|}{\alpha}\sin\angle(\X,x),
\end{eqnarray*}
i.e., relation \eqref{sinUu} holds. \eqref{sinVv} can be proved
similarly.
\end{proof}

Proposition \ref{thm:1} shows that when $\X$ contains good
information on the desired $x$, the qualities of $\UU$ and
$\V$ are determined by $\alpha$, $\|A\|$, $\|x\|$, and by
$\beta$, $\|B\|$, $\|x\|$, respectively.
Since $\|X^{-1}\|^{-1}\leq\|x\|\leq\|X\|$, Theorem 2.3 in
\cite{hansen1989regularization} states that
\begin{equation}\label{normralat}
  \|X\|=\left\|\bsmallmatrix{A\\B}^{\dag}\right\|
  \quad\mbox{and}\quad
  \|X^{-1}\|=\left\|\bsmallmatrix{A\\B}\right\|,
\end{equation}
where $\dag$ denotes the More-Penrose generalized inverse.
Therefore, we have
\begin{equation}\label{boundx}
\left\|\bsmallmatrix{A\\B}\right\|^{-1}\leq \|x\|\leq
\left\|\bsmallmatrix{A\\B}^{\dag}\right\|.
\end{equation}
By the assumption that $\bsmallmatrix{A\\B}$ is well conditioned
and $\bsmallmatrix{A\\B}$ is scaled,
it is clear that $\|x\|$ is not large. Therefore, the qualities of
$\UU$ and $\V$ are similar to that of $\X$ provided that $\alpha$ and $\beta$
are not very small, i.e., $\sigma$ is not very small or large.
The proposition also shows that, for any $\sigma$, at least one of $\UU$ and $\V$
is as good as $\X$ as $\alpha$ and $\beta$ cannot be
small simultaneously.

Given $(A,B)$ and a right searching subspace $\X$, we now propose an
extraction approach that seeks an approximate generalized singular
value pair $(\tilde\alpha,\tilde\beta)$ with $\tilde\alpha^2+\tilde\beta^2=1$
and corresponding approximate
generalized singular vectors $\tilde u\in\mathcal{U}$, $\tilde v
\in\mathcal{V}$ with $\|\tilde u\|=\|\tilde v\|=1$ and
$\tilde x\in\mathcal{X}$ satisfying the
orthogonal projection:
\begin{equation}\label{galerkin}
  \left\{\begin{aligned}
  A\tilde x-\tilde\alpha\tilde u&=\bm{0}, \\
  B\tilde x-\tilde\beta\tilde v&=\bm{0}, \\
  \tilde\beta A^T\tilde u-\tilde\alpha B^T
  \tilde v&\perp\mathcal{X}.
  \end{aligned}\right.
  \end{equation}

Let the columns of $\widetilde X\in\mathbb{R}^{n\times k}$ form an
orthonormal basis of $\mathcal{X}$ and
\begin{equation}\label{qrAXBX}
A\widetilde X=\widetilde UG
\quad\mbox{and}\quad
B\widetilde X=\widetilde VH
\end{equation}
be thin QR factorizations of $A\widetilde X$ and $B\widetilde X$,
respectively, where $G\in\mathbb{R}^{k\times k}$
and $H\in\mathbb{R}^{k\times k}$ are upper triangular. Suppose that $G$ and $H$
are nonsingular.
Then the columns of $\widetilde U\in\mathbb{R}^{m\times k}$ and
$\widetilde V\in\mathbb{R}^{p\times k}$ form orthonormal bases
of $\mathcal{U}$ and $\mathcal{V}$, respectively.
Write $\tilde u=\widetilde Ue$, $\tilde v=\widetilde Vf$ and
$\tilde x=\widetilde Xd$.
Then \eqref{galerkin} is equivalent to
\begin{equation}\label{projected}
  \left\{\begin{aligned}
  Gd&=\tilde \alpha e,\\
  Hd&=\tilde \beta f, \\
  \tilde\beta G^Te &=\tilde\alpha H^Tf.
  \end{aligned}\right.
  \end{equation}
That is, $(\tilde\alpha,\tilde\beta)$ is a generalized singular value
of the $k\times k$ matrix pair $(G,H)$,
and $e$, $f$ and $d$ are the corresponding left and right generalized singular
vectors. We compute the GSVD of $(G,H)$, pick up
$\theta=\frac{\tilde \alpha}{\tilde \beta}$ closest to the target $\tau$, and
take
\begin{equation}\label{approxgsvd}
  (\tilde \alpha,\tilde \beta,
  \tilde u=\widetilde Ue,
  \tilde v=\widetilde Vf,
  \tilde x=\widetilde Xd)
\end{equation}
as an approximation to the desired GSVD component $(\alpha,\beta,u,v,x)$ of
$(A,B)$.

For the accuracy of the approximate left generalized singular vectors
$\tilde u$ and $\tilde v$, notice that $\tilde u$, $u$ and $\tilde v$, $v$ are collinear with $A\tilde x$, $Ax$ and $B\tilde x$, $Bx$, respectively.
Applying \eqref{sinwithA} to $(A\tilde x,Ax)$ and $(B\tilde x, Bx)$,
we obtain the following result.

\begin{propositionmy}\label{thm:1.1}
Let $\tilde u\in\mathcal{U}$, $\tilde v\in\V$ and
$\tilde x\in\X$ be the approximations to the generalized singular vectors $u$,
$v$ and $x$ of $(A,B)$ corresponding to the generalized singular value
$(\alpha,\beta)$ that satisfy \eqref{galerkin}. Then
\begin{eqnarray}
   \sin\angle(\tilde u,u)&\leq&
   \frac{\|A\|\|x\|}{\alpha}
   \sin\angle(\tilde x,x), \label{sinuwideu}\\
   \sin\angle(\tilde v,v)&\leq&
   \frac{\|B\|\|x\|}{\beta}
   \sin\angle(\tilde x,x).  \label{sinvwidev}
\end{eqnarray}
\end{propositionmy}

Proposition~\ref{thm:1.1} indicates that, with the left researching
subspaces $\UU=A\X$ and $\V=B\X$,
our extraction approach \eqref{galerkin} can indeed obtain
the approximate left and right generalized
singular vectors $\tilde u$, $\tilde v$ and $\tilde x$ with similar
accuracy if $\bsmallmatrix{A\\B}$ is well conditioned
and both $\alpha$ and $\beta$ are not very small. Moreover, for any
$\sigma$, at least one of $\tilde u$ and $\tilde v$ is
as accurate as $\tilde x$ as at most one of $\alpha$ and $\beta$
can be small.

It is easily justified that the extraction approach \eqref{galerkin}
mathematically amounts to realizing the standard
orthogonal projection, i.e., the standard Rayleigh--Ritz approximation,
of the regular matrix pair $(A^TA,B^TB)$ onto $\X$. It
extracts the Ritz vector $\tilde x\in\X$ associated with the Ritz value
$\theta^2=\tilde\alpha^2/\tilde\beta^2$ closest to $\tau^2$
and computes $\tilde\alpha$, $\tilde\beta$ and
$\tilde u$, $\tilde v$ satisfying \eqref{galerkin}.
However, we do not form $A^TA$ and $B^TB$ explicitly and
thus avoid the potential accuracy loss of the computed GSVD components.

By \eqref{galerkin},
$(\tilde\alpha,\tilde\beta,\tilde u,\tilde v,\tilde x)$ as an
approximate GSVD component of $(A,B)$ satisfies
$A\tilde x=\tilde\alpha\tilde u$ and $B\tilde x=\tilde\beta\tilde v$,
which lead to $\tilde\alpha=\tilde u^TA\tilde x$ and $\tilde\beta=\tilde v^TB\tilde
x$. Therefore, the (absolute) residual of $(\tilde\alpha,\tilde\beta,\tilde u,\tilde
v,\tilde x)$ is
\begin{equation}\label{residual}
   r=r(\tilde\alpha,\tilde\beta,
   \tilde u,\tilde v,\tilde x)
   :=\tilde\beta A^T\tilde u-
   \tilde\alpha B^T\tilde v.
\end{equation}
It is easily seen that $r=\bm{0}$ if and only if
$(\tilde\alpha,\tilde\beta,\tilde u,\tilde v,\tilde x)$
is an exact GSVD component of $(A,B)$.
We always have
\begin{equation*}
  \|r\|\leq\tilde\beta\|A\|+\tilde\alpha\|B\|,
\end{equation*}
meaning that $\|r\|$ is never large for the scaled $\|A\|$ and $\|B\|$.
In practical computations, for a prescribed tolerance $tol>0$,
$(\tilde \alpha,\tilde \beta,\tilde u,
\tilde v,\tilde x)$ is claimed to have converged if
\begin{equation}\label{outstopping}
  \|r\|\leq
  (\tilde \beta\|A\|_1+\tilde \alpha\|B\|_1)\cdot tol,
\end{equation}
where $\|\cdot\|_1$ denotes the 1-norm of a matrix.

In the following, we present one of the main results, which, in terms of $\|r\|$,
gives the accuracy estimate of the approximate generalized singular value
$\theta=\frac{\tilde{\alpha}}{\tilde{\beta}}$.
To this end and also for our later use, introduce the function
\begin{equation}\label{defineh}
h(\theta,\varsigma)=\frac{\varsigma^2-\theta^2}{1+\varsigma^2}
\qquad\mbox{for}\quad
\theta\geq0
\quad\mbox{and}\quad
\varsigma\geq0.
\end{equation}
By $\alpha_i^2+\beta_i^2=1$ and
$1+\sigma_i^2=\frac{\beta_i^2+\alpha_i^2}{\beta_i^2}=\frac{1}{\beta_i^2}$,
we have
\begin{equation}\label{equiveq}
\alpha_i^2-\beta_i^2 \theta^2=(\sigma_i^2-\theta^2)\beta_i^2
=h(\theta,\sigma_i)
\qquad\mbox{for}\qquad i=1,\dots,q.
\end{equation}

\begin{theoremmy}\label{thm:2}
  Let $(\tilde\alpha,\tilde\beta,\tilde u,
  \tilde v,\tilde x)$ be an approximate GSVD
  component of $(A,B)$ satisfying \eqref{galerkin}
  with $\theta=\frac{\tilde\alpha}{\tilde\beta}$
and $r$ be the corresponding residual defined by \eqref{residual}.
Then the following results hold:
(i) If
\begin{equation}\label{assump1}
\frac{\sigma_{\min}}{\sqrt{2+\sigma_{\min}^2}}<\theta<\sqrt{1+2\sigma_{\max}^2}
\end{equation}
with $\sigma_{\max}$ and $\sigma_{\min}$ the largest and smallest nontrivial
generalized singular values of $(A,B)$, respectively, then there exists a nontrivial
generalized singular value $\sigma$ of $(A,B)$ such that
  \begin{equation}\label{convergence}
    \frac{|\sigma^2-\theta^2|}{(1+\sigma^2)\theta}
    \leq\frac{\|X\|^2\|r\|}{\|\tilde x\|};
  \end{equation}
(ii) if $\theta\geq \sqrt{1+2\sigma_{\max}^2}$, then
\begin{equation}\label{maxs}
\frac{1}{\theta}\leq\frac{\|X\|^2\|r\|}{\|\tilde x\|};
\end{equation}
(iii) if $\theta\leq \frac{\sigma_{\min}}{\sqrt{2+\sigma_{\min}^2}}$, then
\begin{equation}\label{mins}
\theta\leq\frac{\|X\|^2\|r\|}{\|\tilde x\|}.
\end{equation}
\end{theoremmy}

\begin{proof}
By definition \eqref{residual} and
$A\tilde x=\tilde\alpha\tilde u$ and $B\tilde x=\tilde\beta\tilde v$, we have
  \begin{equation}\label{relatrx}
    \theta r=\tilde\alpha A^T\tilde u-
    \tfrac{\tilde\alpha^2}{\tilde\beta}B^T\tilde v
    =(A^TA-\theta^2B^TB)\tilde x.
  \end{equation}
  Premultiplying the two hand sides of the above by $X^T$ and exploiting
  \eqref{gsvd}, we obtain
  \begin{equation*}
  \theta X^Tr
  =X^T(A^TA-\theta^2B^TB)XX^{-1}\tilde x
  =\diag\{C^2-\theta^2S^2,-\theta^2I_{q_1},I_{q_2} \}X^{-1}\tilde x.
  \end{equation*}
Taking norms on the above two hand sides and exploiting
\eqref{equiveq} give
  \begin{eqnarray}
    \theta\| X^Tr\|&\geq&\min \{1,\theta^2,\min_{i=1,\ldots,q}|\alpha_i^2-
    \beta_i^2\theta^2|\}\|X^{-1}\tilde x\|\nonumber\\
    &\geq&\frac{\|\tilde x\|}{\|X\|}\min
    \{1,\theta^2,\min_{i=1,\ldots,q}|h(\theta,\sigma_i)|\} \label{inserteq}.
  \end{eqnarray}

By \eqref{defineh}, for $\theta<\sigma_{\max}$ and
$\sigma_{\max}\leq \theta<\sqrt{1+2\sigma_{\max}^2}$,
we have
\begin{eqnarray*}
|h(\theta,\sigma_{\max})|&=&h(\theta,\sigma_{\max})=
\frac{\sigma_{\max}^2-\theta^2}{1+\sigma_{\max}^2}<1,\\
|h(\theta,\sigma_{\max})|&=&-h(\theta,\sigma_{\max})=
\frac{\theta^2-\sigma_{\max}^2}{1+\sigma_{\max}^2}<1,
\end{eqnarray*}
respectively, proving that $\min\limits_{i=1,2,\ldots,q}|h(\theta,\sigma_i)|<1$.
For $\theta>\sigma_{\min}$ and
$\frac{\sigma_{\min}}{\sqrt{2+\sigma_{\min}^2}}<\theta\leq\sigma_{\min}$,
we obtain
\begin{eqnarray*}
|h(\theta,\sigma_{\min})|&=&-h(\theta,\sigma_{\min})=
\frac{\theta^2-\sigma_{\min}^2}{1+\sigma_{\min}^2}<\theta^2,\\
|h(\theta,\sigma_{\min})|&=&h(\theta,\sigma_{\min})=
\frac{\sigma_{\min}^2-\theta^2}{1+\sigma_{\min}^2}<\theta^2,
\end{eqnarray*}
respectively, proving that  $\min\limits_{i=1,2,\ldots,q}|h(\theta,\sigma_i)|<\theta^2$.
Therefore, under condition \eqref{assump1}, we have
\begin{eqnarray*}
\min\{1,\theta^2,\min_{i=1,\ldots,q}|h(\theta,\sigma_i)|\}
&=&\min_{i=1,\ldots,q}|h(\theta,\sigma_i)|=|h(\theta,\sigma)|,
\end{eqnarray*}
which, together with \eqref{inserteq}, proves \eqref{convergence}.

It is straightforward to
justify that, under the conditions in (ii) and (iii),
\begin{eqnarray*}
\min\{1,\theta^2,\min_{i=1,\ldots,q}|h(\theta,\sigma_i)|\}
&=&1,\\
\min\{1,\theta^2,\min_{i=1,\ldots,q}|h(\theta,\sigma_i)|\}&=&\theta^2,
\end{eqnarray*}
respectively. Therefore, from \eqref{inserteq} we obtain
\eqref{maxs} and \eqref{mins}.
\end{proof}

By assumption and \eqref{normralat}, $\|X\|$ is modest.
From \eqref{qrAXBX} and the orthonormality of $\widetilde{X}$,
it is easily justified that
\begin{equation}\label{boundproj}
\left\|\bsmallmatrix{G\\H}\right\|\leq
\left\|\bsmallmatrix{A\\B}\right\|,\ \ \left\|\bsmallmatrix{G\\H}^{\dag}
\right\|\leq\left\|\bsmallmatrix{A\\B}^{\dag}\right\|.
\end{equation}
Exploiting Lemma 2.4 of \cite{huang2020choices}
and adapting \eqref{normralat} and \eqref{boundx} to $(G,H)$,
for $d$ defined by \eqref{projected} we obtain
\begin{equation}\label{boundd}
\left\|\bsmallmatrix{G\\H}\right\|^{-1}\leq \|d\|\leq
\left\|\bsmallmatrix{G\\H}^{\dag}\right\|.
\end{equation}
Therefore, from \eqref{normralat}, \eqref{boundproj}, \eqref{boundd}
and $\|\tilde x\|=\|\tilde X d\|=\|d\|$ we have
$$
\frac{\|X\|^2}{\|\tilde x\|}\leq \left\|\bsmallmatrix{A\\B}^{\dag}\right\|
\kappa\left(\bsmallmatrix{A\\B}\right),
$$
which is modest when $\bsmallmatrix{A\\B}$
is scaled and well conditioned. Relations~\eqref{maxs}
and \eqref{mins} show that if $\theta$ is significantly
bigger than $\sigma_{\max}$ or smaller than $\sigma_{\min}$ then it
converges to the trivial zero or infinite generalized
singular value as $\|r\|$ tends to zero.

For the scaled matrix pair $(\gamma A,\gamma B)$ with
the constant $\gamma>0$, the approximate GSVD components
becomes $(\tilde\alpha,\tilde\beta,\tilde u,\tilde v,
\frac{1}{\gamma}\tilde x)$, the residual is $\gamma r$, and
the right generalized singular vector matrix of
$(\gamma A,\gamma B)$ is $\frac{1}{\gamma}X$.
Inserting them into \eqref{convergence}--\eqref{mins},
we obtain the same results. These indicate that
$\frac{\|X\|^2\|r\|}{\|\tilde x\|}$ in the right-hand
sides of \eqref{convergence}--\eqref{mins} is invariant under the scaling
of $(A,B)$.
Notice that $\|r\|$ is a backward error
and the left-hand side of \eqref{convergence}
is a forward error of $\sigma$. It is instructive to regard the factor
$\frac{\|X\|^2}{\|\tilde x\|}$ as a condition
number of $\sigma$ when bounding the error of
$\sigma$ in terms of the residual norm $\|r\|$.
Moreover, if $\theta\approx 0$ but $\theta> \frac{\sigma_{\min}}{\sqrt{2+\sigma_{\min}^2}}$ and
$\theta>1$ but $\theta<\sqrt{1+2\sigma_{\max}^2}$, from \eqref{convergence}
we approximately have the absolute errors $|\sigma-\theta|\lesssim\frac{\|X\|^2}
{2\|\tilde x\|}\|r\|$ and
$|\sigma-\theta|\lesssim\frac{\theta^2\|X\|^2}{2\|\tilde x\|}\|r\|$, respectively. Therefore, \eqref{convergence} indicates
that $|\sigma-\theta|=\mathcal{O}(\|r\|)$ with a
generic constant in the big $\mathcal{O}(\cdot)$.

\subsection{Subspace expansion}\label{sec:3}

If the current GSVD approximation $(\tilde\alpha,\tilde\beta,
\tilde u,\tilde v,\tilde x)$ does not yet converge, one needs to
expand the searching subspaces $\X,\mathcal{U},\V$ in order to extract
a more accurate approximate GSVD component with respect to them.
Since we construct the left searching subspaces by $\mathcal{U}=A\X$ and
$\mathcal{V}=B\X$, we only need to expand  $\mathcal{X}$
effectively and then generate $\mathcal{U}=A\X$
and $\mathcal{V}=B\X$ correspondingly.

Keep in mind that $(\sigma^2,x)$ is an eigenpair of
the matrix pair $(A^TA,B^TB)$ with $\sigma=\alpha/\beta$. Suppose that
an approximate right generalized singular vector
$\tilde x\in \X$ is available. We aim to seek a
correction vector $t$ satisfying
\begin{equation}\label{require}
   t\perp \tilde y:=(A^TA+B^TB)\tilde x
   =\tilde\alpha A^T\tilde u+\tilde\beta B^T\tilde v,
\end{equation}
such that $\tilde x+t$ is an unnormalized right generalized singular
vector of $(A,B)$. Therefore, $(\sigma^2,\tilde x+t)$ is an
exact eigenpair of $(A^TA,B^TB)$:
\begin{equation}\label{innerhope}
  A^TA(\tilde x+t)=\sigma^2 B^TB(\tilde x+t).
\end{equation}
Rearranging this equation, we obtain
\begin{equation}\label{eqinnerhope}
  (A^TA-\theta^2 B^TB)t=
  -(A^TA-\theta^2 B^TB)\tilde x
  +(\sigma^2-\theta^2)B^TB\tilde x
  +(\sigma^2-\theta^2)B^TBt,
\end{equation}
where $\theta=\frac{\tilde\alpha}{\tilde\beta}$
is the current approximate generalized singular value.

Assume that $\tilde x$ is already reasonably accurate with the normalization
$\tilde x^T (A^TA+B^TB)\tilde x=1$, which means that
$\|t\|$ is small relative to $\|\tilde x\|$.
In this case, $\theta$ is an approximation to $\sigma$ with the error
$\mathcal{O}(\|t\|)$ because
\begin{equation}\label{errorsig}
   \sigma^2=\frac{\|A(\tilde x+t)\|^2}{\|B(\tilde x+t)\|^2}
   =\frac{\tilde\alpha^2+2\tilde\alpha\tilde u^TAt +\|At\|^2}
   {\tilde\beta^2+2\tilde\beta\tilde v^TBt+\|Bt\|^2}
  =\theta^2(1+\mathcal{O}(\|t\|)),
\end{equation}
indicating that the size of the third term in the right-hand side of
\eqref{eqinnerhope} is $\OO(\|t\|^2)$.

Note from \eqref{relatrx} that the first term in the right-hand
side of \eqref{eqinnerhope} is collinear with  the residual $r$ of
$(\tilde\alpha,\tilde\beta,\tilde u,\tilde v,\tilde x)$,
which is orthogonal to $\mathcal{X}$,
as indicated by the third condition in \eqref{galerkin}.
Therefore, the first term in the right-hand side of
\eqref{eqinnerhope} is orthogonal to $\tilde x\in\mathcal{X}$.
Moreover, we know from \eqref{galerkin} and
\eqref{require} that $\tilde y^T\tilde x=1$
and $(I-\tilde y\tilde x^T)$ is an oblique projector onto
the orthogonal complement $\tilde x^{\perp}$ of $span\{\tilde x\}$.
Neglecting the third term $\OO(\|t\|^2)$ in the right-hand side
of \eqref{eqinnerhope}, we obtain
\begin{equation}\label{precorrection1}
   \left(I-\tilde y\tilde x^T\right)(A^TA-\theta^2 B^TB)t
   =-\theta r+(\theta^2-\sigma^2)(I-\tilde y\tilde x^T)B^TB\tilde x.
\end{equation}
From \eqref{require} and \eqref{innerhope}, we have
\begin{eqnarray*}
(I-\tilde y\tilde x^T)B^TB\tilde x
&=&B^TB\tilde x-(\tilde xB^TB\tilde x)\tilde y
=B^TB\tilde x-\tilde\beta^2\tilde y \\
&=&\tilde\alpha^2B^TB\tilde x-\tilde\beta^2A^TA\tilde x
=\frac{\theta^2B^TB\tilde x-A^TA\tilde x}{1+\theta^2}\\
&=&\frac{(\theta^2-\sigma^2)B^TB\tilde x}{1+\theta^2}
+\frac{(A^TA-\sigma^2B^TB)t}{1+\theta^2}=\OO(\|t\|),
\end{eqnarray*}
which, together with \eqref{errorsig}, proves
that the second term in the right-hand side
of \eqref{precorrection1} is the higher order small $\OO(\|t\|^2)$
relative to the left-hand side of \eqref{precorrection1} and thus the first term of the right-hand side of \eqref{precorrection1}.
Neglecting the second term in the right-hand side of \eqref{precorrection1}, we obtain
\begin{equation}\label{precorrection}
   \left(I-\tilde y\tilde x^T\right)(A^TA-\theta^2 B^TB)t=-\theta r
   \quad\mbox{with}\quad t\perp \tilde y.
\end{equation}

The requirement $t\perp \tilde y$ means
$t=(I-\tilde x\tilde y^T)t$.
Therefore, we can replace $t$ with $(I-\tilde x\tilde y^T)t$ in
\eqref{precorrection}. Notice that
it is the direction other than the size of $t$ that matters
when expanding $\mathcal{X}$ by adding $t$ to it. Therefore,
it makes no difference when solving \eqref{precorrection}
with the right-hand side $-\theta r$ or $-r$.
As a consequence, we have ultimately derived an correction equation
\begin{equation}\label{correction}
  (I-\tilde y\tilde x^T)(A^TA-\theta^2 B^TB)(I-\tilde x\tilde y^T)t=-r
 \quad \mbox{ with} \quad
 t\perp\tilde y.
\end{equation}
Solving it for $t$ and orthonormalizing $t$ against $\widetilde X$ yields the
subspace expansion vector
$x_{+}=\frac{(I-\widetilde X\widetilde X^T)t}
{\|(I-\widetilde X\widetilde X^T)t\|}$.
The $k+1$ columns of the updated
$\widetilde X:=[\widetilde X,x_{+}]$ form an
orthonormal basis of the expanded $(k+1)$-dimensional right searching subspace
$\X:=\X+{\rm span}\{x_+\}$.

The coefficient matrix in \eqref{correction} dynamically depends
on $\theta$ as the outer iterations proceed. In practical computations,
it is typical that $\theta$ may have little accuracy as approximations to $\sigma$ in an initial stage, so that solving \eqref{correction} with varying $\theta$ may not gain. To this end, a better way is to solve the correction equation
\eqref{correction} with $\theta$ replaced by the fixed target $\tau$ in the left-hand side:
\begin{equation}\label{correction2}
(I-\tilde y\tilde x^T)(A^TA-\tau^2 B^TB)(I-\tilde x\tilde y^T)t=-r
\quad\mbox{with}\quad t\perp\tilde y
\end{equation}
in the initial stage and then switch to solving \eqref{correction} when
$\|r\|$ becomes fairly small,
i.e., $\theta$ has already some accuracy.
Approximately solving this equation or \eqref{correction} iteratively is called
the inner iterations in CPF-JDGSVD. In computations, if
\begin{equation}\label{fixtarget}
  \|r\|\leq(\tilde\beta\|A\|_1+
  \tilde\alpha\|B\|_1)\cdot fixtol
\end{equation}
with $fixtol$ fairly small but bigger than
the stopping tolerance $tol$ of outer iterations, we then
switch to solving \eqref{correction}.

\section{Properties of the correction equations and stopping criteria
for the inner iterations}\label{sec:4}

For the large $A$ and $B$, suppose that only iterative
solvers are computationally viable to solve the correction
equations approximately.
Since the coefficient matrices in \eqref{correction} and
\eqref{correction2} are symmetric and typically indefinite,
the minimal residual method (MINRES) is a most commonly used
choice \cite{greenbaum,saad2003}.
We establish upper bounds for the condition
numbers of the correction equations \eqref{correction} and \eqref{correction2}
when $\theta=\sigma$ and $\tilde{x}=x$. Meanwhile,
we make an analysis on the solution accuracy requirement of
the correction equations for practical use.
Based on them, we propose practical stopping criteria for the
inner iterations.
We focus on \eqref{correction}, and, as it will turn out,
the results are directly applicable to \eqref{correction2}.

\subsection{Conditioning}\label{subsec:1}
The coefficient matrix in the correction equation \eqref{correction}
maps the orthogonal complement $\tilde y^{\perp}$ of
$span\{\tilde y\}$ to $\tilde x^{\perp}$, and
it is restricted to $\tilde y^{\perp}$ and
generates elements in $\tilde x^{\perp}$.
We denote this restricted linear operator by
\begin{equation}\label{defM}
    M=(A^TA-\theta^2 B^TB)
    |_{\tilde y^{\perp}\rightarrow \tilde x^{\perp}}.
  \end{equation}
As will be clear, the condition number $\kappa(M)$ determines the reliability of
adopting the relative residual norm of the correction equation
as the measurement of inner iteration accuracy.
As a result, it is significant to derive sharp estimates for $\kappa(M)$.
However, it is generally not possible to do so for
a general approximation $\tilde x$.
Fortunately, sharp estimates for the ideal case that $\theta=\sigma$ and
$\tilde x=x$ suffice since,
by a continuity argument, they will exhibit the
asymptotic behavior of $\kappa(M)$ when $\theta\rightarrow \sigma$
and $\tilde x\rightarrow x$.

Based on the GSVD \eqref{gsvd} of $(A,B)$ and \eqref{gsvalue}, we partition
\begin{equation}\label{partitionCSX}
  U=[u,U_2],\quad
  V=[v,V_2],\quad
  X=[x,X_2],\quad
  \Sigma_{A}=\bsmallmatrix{\alpha&\\&\Sigma_{A,2}},\quad
  \Sigma_{B}=\bsmallmatrix{\beta&\\&\Sigma_{B,2}},
\end{equation}
where the matrices
\begin{equation}\label{SigmaA2B2}
\Sigma_{A,2}=\diag\{C_2,\bm{0}_{l_1,q_1},I_{q_2}\}
\quad\mbox{and}\quad
\Sigma_{B,2}=\diag\{S_2,I_{q_1},\bm{0}_{l_2,q_2}\}
\end{equation}
with $C_2=\diag\{\alpha_2,\dots,\alpha_q\}$ and
$S_2=\diag\{\beta_2,\dots,\beta_q\}$.
From $X^T(A^TA+B^TB)X=I$, we obtain
\begin{equation}\label{partitionY}
  Y=X^{-T}=(A^TA+B^TB)X=[y,Y_2]
\end{equation}
with $y=(A^TA+B^TB)x$ and $Y_2=(A^TA+B^TB)X_2$.
Then $X_2$ and $Y_2$ are orthogonal to $y$ and $x$, respectively,
i.e., $X_2^Ty=\bm{0}$ and $Y_2^Tx=\bm{0}$, and
the columns of $Y_2$ form a basis of
the orthogonal complement $x^{\perp}$ of $span\{x\}$ with respect to $\mathbb{R}^n$.
Let
\begin{equation}\label{qrX2W2}
  Y_2=Q_{y}R_{y}
\end{equation}
be the thin QR factorization of $Y_2$.
Then the columns of $Q_{y}$ form an orthonormal basis of $x^{\perp}$.
It is obvious from \eqref{require} and \eqref{partitionY} that $\tilde y=y$ when
$\tilde x=x$.

With the above preparation and notations, we can present the following results.

\begin{theoremmy}\label{thm:3}
Set $\tilde x=x$ and $\tilde y=y$ in \eqref{defM},
and assume that $\sigma$ is a simple
nontrivial generalized singular value of $(A,B)$.
Then
\begin{equation}\label{expressionM}
M^{\prime}= X_2^T(A^TA-\sigma^2 B^TB)X_2=\Sigma_{A,2}^T\Sigma_{A,2}-
\sigma^2\Sigma_{B,2}^T\Sigma_{B,2}
\end{equation}
is nonsingular, where $X_2$ and $\Sigma_{A,2}$,  $\Sigma_{B,2}$ are defined by \eqref{partitionCSX} and \eqref{SigmaA2B2}, respectively.
Furthermore, it holds that
 \begin{equation}\label{kappaM}
 \kappa(M)=\kappa(R_y M^{\prime}R_y^T)\leq\kappa^2
    \left(\begin{bmatrix}\begin{smallmatrix}A\\B
\end{smallmatrix}\end{bmatrix}\right)
\max\left\{\frac{\max\{1,\sigma^2\}(1+\sigma_{*}^2)}
{|\sigma_{*}^2-\sigma^2|},\frac{\max\{1,\sigma^2\}}{\min\{1,\sigma^2\}}\right\},
\end{equation}
where $R_y$ is as defined in \eqref{qrX2W2} and
$\sigma_{*}$ is the minimizer of
$\min_{i=2,3,\ldots,q}|\alpha_i^2-\beta_i^2\sigma^2|
=\min_{i=2,3,\ldots,q}\frac{|\sigma_i^2-\sigma^2|}{1+\sigma_i^2}$.
\end{theoremmy}

\begin{proof}
For $\tilde x=x$ and $\tilde y=y$, we have $\theta=\sigma$ in \eqref{defM}.
For $Y=X^{-T}$ in \eqref{partitionY}, we have
$YX^T=XY^T=I.$ Therefore, from \eqref{qrX2W2} and
\eqref{expressionM} it follows that the coefficient matrix in \eqref{correction}
is
\begin{align}
  (I-yx^T)(A^TA-\sigma^2 B^TB)(I-xy^T)
 & =  (YX^T-yx^T)(A^TA-\sigma^2 B^TB)(XY^T-xy^T) \notag\\
 &=Y_2X_2^T(A^TA-\sigma^2 B^TB)X_2Y_2^T \notag\\
 &=Q_yR_yM^{\prime}R_y^TQ_y^T. \label{eqoper}
\end{align}
By the GSVD \eqref{gsvd} of $(A,B)$ and \eqref{SigmaA2B2}, it is
straightforward to obtain
\begin{eqnarray}\label{eigMrelat}
M^{\prime}&=&X_2^TA^TAX_2-\sigma^2 X_2^TB^TBX_2 \nonumber\\
&=&\Sigma_{A,2}^T\Sigma_{A,2}-\sigma^2\Sigma_{B,2}^T\Sigma_{B,2}\nonumber\\
&=&\diag\{C_2^2-\sigma^2 S_2^2,-\sigma^2 I_{q_1},I_{q_2}\}.
\end{eqnarray}
As a result, by the assumption, $M^{\prime}$ is nonsingular. Since $Q_y$ is column
orthonormal, it follows from \eqref{defM} and \eqref{eqoper} that $M$ is
nonsingular and
\begin{equation}\label{kappaMbound}
\kappa(M)=\kappa(R_yM^{\prime}R_y^T)\leq \kappa^2(R_y)\kappa(M^{\prime}).
\end{equation}

Notice that $Y_2$ consists of the second to the last columns of
$Y$. Then from \eqref{qrX2W2} and $Y=X^{-T}$ we obtain
\begin{equation*}
\|R_{y}\|=\|Y_2\|\leq\|Y\|=\|X^{-1}\|,\qquad
\|R_{y}^{-1}\|=\|Y_2^{\dag}\|\leq\|Y^{-1}\|=\|X\|, \label{kappaR1}
\end{equation*}
which means that $\kappa(R_{y})\leq\kappa(X)$.
Therefore, it follows from \eqref{normralat} that
\begin{equation}\label{kappaR1R2}
  \kappa(R_{y}) \leq \kappa(\bsmallmatrix{A\\B}).
\end{equation}

From \eqref{equiveq}, the diagonal elements of $M^{\prime}$ are $\alpha_i^2-\beta_i^2\sigma^2=h(\sigma,\sigma_i)$, $i=2,\dots,q$ with
$h(\sigma,\sigma_i)$ defined by \eqref{defineh}. By definition,
it is straightforward that
\begin{equation}\label{upperbound}
|h(\sigma,\sigma_i)|\leq\max\left\{\frac{\sigma_i^2}{1+\sigma_i^2},
\frac{\sigma^2}{1+\sigma_i^2}\right\}
\leq \max\{1,\sigma^2\}.
\end{equation}
Applying it and \eqref{equiveq} to \eqref{eigMrelat} yields
\begin{eqnarray}
 \sigma_{\max}(M^{\prime})
  &=&\max_{i=2,\ldots,q}\{1,\sigma^2,|\alpha_i^2-\beta_i^2\sigma^2|\}
  \leq\max\{1,\sigma^2\}. \label{maxmp}
\end{eqnarray}

Note that
$$
 \sigma_{\min}(M^{\prime})\!=\!
 \min\{1,\sigma^2,\min_{i=2,\ldots,q}|\alpha_i^2-\beta_i^2\sigma^2|\}.
$$
We next consider the following two cases.

Case (i): If
\begin{equation*}
 \sigma_{\min}(M^{\prime})\!=\!
\!\min_{i=2,\ldots,q}|\alpha_i^2-\beta_i^2\sigma^2|
  \! =\!|\alpha_*^2-\beta_*^2\sigma^2|=\frac{|\sigma_{*}^2-\sigma^2|}
   {1+\sigma_{*}^2},
\end{equation*}
that is, $\sigma$ is comparatively clustered with $\sigma_*$, then
by \eqref{maxmp} we obtain
\begin{equation}\label{kappacprime}
  \kappa(M^{\prime})=
\frac{\max\{1,\sigma^2\}(1+\sigma_{*}^2)}{|\sigma_{*}^2-\sigma^2|}.
\end{equation}

Case (ii): If
\begin{equation*}
 \sigma_{\min}(M^{\prime})\!=\!\min\{1,\sigma^2\},
\end{equation*}
that is, $\sigma$ is comparatively well separated from $\sigma_*$, then
it follows from \eqref{maxmp} that
$$
\kappa(M^{\prime})=
\frac{\max\{1,\sigma^2\}}{\min\{1,\sigma^2\}}.
$$
Relation \eqref{kappaM} follows from applying this relation,
\eqref{kappacprime} and \eqref{kappaR1R2} to \eqref{kappaMbound}.
\end{proof}

Theorem~\ref{thm:3} indicates that $\kappa(M)$ is bounded by
$\kappa^2(\bsmallmatrix{A\\B})$ multiplying
the first one in the maximum term of \eqref{kappaM}
if $\sigma$ is comparatively clustered with some other $\sigma_i$
and by $\kappa^2(\bsmallmatrix{A\\B})$ multiplying
the second one in the maximum term of \eqref{kappaM}
if $\sigma$ is comparatively well separated
from all the other $\sigma_i$, $i=2,3,\ldots,q$.

In an analogous manner, for the correction equation \eqref{correction2},
define $M_{\tau}$ by replacing $\theta$ with $\tau$ in \eqref{defM}.
Then for $\tilde x=x$ and $\tilde y=y$, we have
\begin{equation}\label{kappaMp}
    \kappa(M_{\tau})\leq\kappa^2
    \left(\begin{bmatrix}\begin{smallmatrix}A\\B
\end{smallmatrix}\end{bmatrix}\right)
\max\left\{\frac{\max\{1,\tau^2\}(1+\sigma_{*,\tau}^2)}
{|\sigma_{*,\tau}^2-\tau^2|},\frac{\max\{1,\tau^2\}}{\min\{1,\tau^2\}}\right\},
\end{equation}
where $\sigma_{*,\tau}$ is the minimizer of
$\min_{i=2,3,\ldots,q}\frac{|\sigma_i^2-\tau^2|}{1+\sigma_i^2}$.

We remark that, by a continuity argument,
bounds \eqref{kappaM} and \eqref{kappaMp} asymptotically hold as
$\tilde x\rightarrow x$ and $\theta\rightarrow \sigma$. Therefore,
\eqref{kappaM} and \eqref{kappaMp} give good estimates
for $\kappa(M)$ and $\kappa(M_{\tau})$, respectively,
once $\tilde x$ becomes a reasonably good approximation to $x$.

\subsection{Accuracy requirements on the inner iterations}\label{subsec:2}

We make an analysis on the inner iterations and establish some robust accuracy requirements on them, so that the outer iterations
of the resulting CPF-JDGSVD mimic the exact counterpart of
CPF-JDGSVD where the correction equations are solved accurately.

Assume that $\theta$ is not a generalized singular value of $(A,B)$.
Then the matrix $A^TA-\theta^2B^TB$ is nonsingular. Denote the matrices
\begin{equation}\label{defLK}
  L=(A^TA-\theta^2B^TB)^{-1}
  \quad\mbox{and}\quad
  K=L(A^TA+B^TB).
\end{equation}
The eigenpairs $(\sigma^2,x)$, $(0,x)$ and $(+\infty,x)$
of $(A^TA,B^TB)$ are transformed into
the eigenpairs $(\frac{\sigma^2+1}{\sigma^2-\theta^2},x)$,
$(-\frac{1}{\theta^2},x)$ and $(1,x)$ of $K$, respectively.

By \eqref{relatrx} and $(I-\tilde x\tilde y^T)t=t$,
equation \eqref{correction} can be rearranged as
\begin{equation*}
   (I-\tilde y\tilde x^T)L^{-1} (\theta t)=-L^{-1}\tilde x,
\end{equation*}
whose solution is
\begin{equation}\label{expressiony}
  \theta t=-\tilde x+\nu L\tilde y
\end{equation}
with $\nu=\tilde x^TL^{-1}(\theta t)$.
Premultiplying both hand sides of \eqref{expressiony} by $\tilde y^T$ and making
use of the orthogonality $t\perp \tilde y$ and
the normalization $\tilde y^T\tilde x=1$, we obtain
\begin{equation}\label{defnu}
\nu=\frac{1}{\tilde y^TL\tilde y}.
\end{equation}

Let $\tilde t$ be an approximate solution of \eqref{correction} with the relative
error $\varepsilon=\frac{\|\tilde t- t\|}{\|t\|}$. Then $\tilde t$ can be written
as
\begin{equation}\label{widey}
  \tilde t=t+\varepsilon\|t\|s,
\end{equation}
where $s$ is the error direction vector with $\|s\|=1$,
and the exact and inexact expansion vectors are
$x_{+}=\frac{(I-\widetilde X\widetilde X^T)t}
{\|(I-\widetilde X\widetilde X^T)t\|}$ and
$\widetilde x_{+}=\frac{(I-\widetilde X\widetilde X^T)\tilde t}
{\|(I-\widetilde X\widetilde X^T)\tilde t\|}$, respectively.
The relative error of $x_{+}$ and $\widetilde x_{+}$ can be
defined as
\begin{equation}\label{widevarepsilon}
  \tilde\varepsilon=
  \frac{\|(I-\widetilde X\widetilde X^T)\tilde t
  -(I-\widetilde X\widetilde X^T)t\|}
  {\|(I-\widetilde X\widetilde X^T)t\|}.
\end{equation}

As has been shown in \cite{huang2019inner,jia2014inner,jia2015harmonic},
in order to make the ratio of the distance between $x$ and $\X+span\{x_+\}$ and that between
$x$ and $\X+span\{\tilde{x}_+\}$ lie in
$[0.999,1.001]$, which means that $\X+span\{x_+\}$ and $\X+span\{\tilde{x}_+\}$ contain
almost the same information on $x$, it generally suffices to take a fairly small
\begin{equation}\label{range}
\tilde\varepsilon\in[10^{-4},10^{-3}],
\end{equation}
which will be utilized when designing robust stopping criteria for
the inner iterations.

The following result establishes an intimate relationship
between $\varepsilon$ and
$\tilde\varepsilon$.

\begin{theoremmy}\label{thm:4}
  Let $\tilde t$ be an approximation to the exact solution $t$ of
  \eqref{correction} with the relative error $\varepsilon$ satisfying
  \eqref{widey}, and $\tilde\varepsilon$ be defined by \eqref{widevarepsilon}.
  Let the matrices $L$ and $K$ be defined by \eqref{defLK} and
  $K^{\prime}=X_{\perp}^TKX_{\perp}$ with $X_{\perp}$
  such that $[\frac{x}{\|x\|},X_{\perp}]$ is orthogonal, and
  assume that
  \begin{equation}\label{sep}
    \sep(\rho,K^{\prime})=\|(K^{\prime}-\rho I)^{-1}\|^{-1}>0
    \qquad\mbox{with}\qquad \rho=\tilde y^TL\tilde y=\frac{1}{\nu}.
  \end{equation}
  Then
  \begin{equation}\label{inaccrelat}
    \varepsilon\leq\frac{2\|K\|}{\sep(\rho,K^{\prime})
    \|s_{\perp}\|}\tilde\varepsilon,
  \end{equation}
  where $s_{\perp}=(I-\widetilde X\widetilde X^T)s$ with the vector $s$
  defined by \eqref{widey}.
\end{theoremmy}

\begin{proof}
Premultiplying both hand sides
of \eqref{widey} by $(I\!-\!\widetilde X\widetilde X^T)$ and taking norms
give
  \begin{equation*}
    \varepsilon=\frac{\|(I-\widetilde X\widetilde X^T)\tilde t
    -(I-\widetilde X\widetilde X^T)t\|}
    {\|t\|\|(I-\widetilde X\widetilde X^T)s\|}
    =\frac{\|(I-\widetilde X\widetilde X^T)t\|}
    {\|t\|\|s_{\perp}\|}\tilde\varepsilon.
  \end{equation*}
  By \eqref{expressiony}, substituting $t=
  \frac{1}{\theta}(-\tilde x+\nu L\tilde y)$ into the above relation
  and making use of $\tilde x\in\X$ and
  $\tilde y=(A^TA+B^TB)\tilde x$, we obtain
  \begin{eqnarray}
  \varepsilon
  &=&\frac{\|(I-\widetilde X\widetilde X^T)(-\tilde x+\nu L\tilde y)\|}
    {\|-\tilde x+\nu L\tilde y\|\|s_{\perp}\|}\tilde\varepsilon\
  =\frac{\|(I-\widetilde X\widetilde X^T)(\nu L\tilde y)\|}
    {\|\nu L\tilde y-\tilde x\|\|s_{\perp}\|}\tilde\varepsilon  \nonumber\\
  &=&\frac{\|(I-\widetilde X\widetilde X^T)
  (\nu L(A^TA+B^TB)\tilde x)\|}
    {\|\nu L(A^TA+B^TB)\tilde x-\tilde x\|\|s_{\perp}\|}\tilde\varepsilon \nonumber\\
  &=&\frac{\|(I-\widetilde X\widetilde X^T)K
  \tilde x\|}{\|K\tilde x-\rho
  \tilde x\|\|s_{\perp}\|}\tilde\varepsilon,
  \label{varepsilon}
  \end{eqnarray}
  where, by definition \eqref{defnu}, $\rho=\frac{1}{\nu}=\tilde y^TL\tilde y$.

Since $\tilde x\rightarrow x$ and $\tilde y\rightarrow y=(A^TA+B^TB)x$,
from \eqref{defLK} and $x^Ty=1$ we have
\begin{equation}\label{rhoapprox}
   \rho \rightarrow  y^TL(A^TA+B^TB)x
   =y^TKx
    = \frac{\sigma^2+1}{\sigma^2-\theta^2}y^Tx
   =\frac{\sigma^2+1}{\sigma^2-\theta^2}.
 \end{equation}
 Therefore, the pair $(\rho,\tilde x)$ is an
 approximation to the simple eigenpair
 $(\frac{\sigma^2+1}{\sigma^2-\theta^2},x)$ of $K$.
Set $x_s=x/\|x\|$, and notice by assumption that $[x_s,X_{\perp}]$ is orthogonal.
Then we have a Schur like decomposition:
 \begin{equation}\label{eigK}
   \begin{bmatrix}x_s^T\\X_{\perp}^T\end{bmatrix}K
   \begin{bmatrix}x_s&X_{\perp}\end{bmatrix}=
   \begin{bmatrix}\frac{\sigma^2+1}{\sigma^2-\theta^2} &x_s^TKX_{\perp}
   \\\bm{0}&K^{\prime}\end{bmatrix}
 \end{equation}
 with $K^{\prime}=X_{\perp}^TKX_{\perp}$.
 By Theorem 6.1 of \cite{jia2001analysis}, we obtain
 \begin{equation}\label{denominator}
 \|K\tilde x_s-\rho\tilde x_s\|\geq
 \sin\psi\cdot\sep(\rho,K^{\prime}),
 \end{equation}
 where $\sep(\rho,K^{\prime})=\|(K^{\prime}-\rho I)^{-1}\|^{-1}$
 and $\psi=\angle(\tilde x,x)$
 is the acute angle of $\tilde x$ and $x$.

Let us decompose $\tilde x_s$ and $x_s$ into the orthogonal
 direct sums:
 \begin{equation*}
 \tilde x_s=x_s\cos\psi+w_1\sin\psi
 \quad\mbox{and}\quad
 x_s=\tilde x_s\cos\psi+w_2\sin\psi,
 \end{equation*}
where $w_1\perp x_s$ and $w_2\perp \tilde x_s$ with $\|w_1\|=\|w_2\|=1$.
Exploiting these two decompositions and
$(I-\widetilde X\widetilde X^T)\tilde x_s=\bm{0}$ yields
\begin{eqnarray*}
 (I-\widetilde X\widetilde X^T)K\tilde x_s
 &=&(I-\widetilde X\widetilde X^T)
    K(x_s\cos\psi+w_1\sin\psi)\\
 &=&(I-\widetilde X\widetilde X^T)
    \left(\frac{\sigma^2+1}{\sigma^2-\theta^2}
     x_s\cos\psi+Kw_1\sin\psi\right)\\
 &=&(I-\widetilde X\widetilde X^T)
    \left(\frac{\sigma^2+1}{\sigma^2-\theta^2}
    (\tilde x_s\cos\psi+w_2\sin\psi)
    \cos\psi+Kw_1\sin\psi\right)   \\
 &=&(I-\widetilde X\widetilde X^T)
    \left(\frac{\sigma^2+1}{\sigma^2-\theta^2}
    w_2\cos\psi+Kw_1\right)\sin\psi.
\end{eqnarray*}
Since $\frac{\sigma^2+1}{\sigma^2-\theta^2}$ is an eigenvalue of $K$, we have
$\frac{\sigma^2+1}{|\sigma^2-\theta^2|}\leq\|K\|$.
Taking norms on both hand sides in the above
relation, we obtain
\begin{eqnarray}
  \|(I-\widetilde X\widetilde X^T)K\tilde x_s\|
  &\leq&\|I-\widetilde X\widetilde X^T\|\|K\|
  (\|w_2\|\cos\psi+\|w_1\|)\sin\psi\nonumber \\
  &\leq&2\|I-\widetilde X\widetilde X^T\|\|K\|\sin\psi \nonumber \\
  &\leq&2\|K\|\sin\psi. \label{numerator}
\end{eqnarray}
Relation \eqref{inaccrelat} then follows by applying \eqref{denominator}
and \eqref{numerator} to \eqref{varepsilon}.
\end{proof}

Theorem \ref{thm:4} reveals an intrinsic connection between the
solution accuracy $\varepsilon$ of the correction equation
\eqref{correction} and the
accuracy $\tilde\varepsilon$ of the expansion vector
$\widetilde{x}_{+}$. For the solution accuracy
$\varepsilon$ of the correction equation \eqref{correction2} and the
accuracy $\tilde\varepsilon$ of the corresponding expansion vector,
we can analogously prove
\begin{equation}\label{inaccrelat2}
    \varepsilon\leq\frac{2\|K_{\tau}\|}{\sep(\rho_{\tau},K_{\tau}^{\prime})
    \|s_{\perp}\|}\tilde\varepsilon
    \quad\mbox{with}\quad
    \rho_{\tau}=\tilde y^TL_\tau \tilde y,
\end{equation}
where
\begin{equation}\label{taupair}
L_{\tau}=(A^TA-\tau^2B^TB)^{-1}, \quad
K_{\tau}=L_{\tau}(A^TA+B^TB),\quad
K_{\tau}^{\prime}=X_{\perp}^TK_{\tau}X_{\perp}.
\end{equation}

Relation \eqref{inaccrelat} (resp. \eqref{inaccrelat2}) indicates that
once $\tilde\varepsilon$ is given, we are able to determine the $least$ or
$lowest$ accuracy requirement $\varepsilon$ for the correction
equation \eqref{correction} (resp. \eqref{correction2}) from
\eqref{inaccrelat} (resp. \eqref{inaccrelat2}).

\subsection{Stopping criteria for the inner iterations}\label{subsec:3}

Our goal is to practically derive the least accuracy requirement for the
approximate solution of the relevant correction equation, so that
the resulting (inexact) CPF-JDGSVD method and the
exact CPF-JDGSVD method where the correction equations are solved accurately
use almost the same outer iterations
to achieve a prescribed stopping tolerance. We next show how to 
design practical stopping criteria for the
inner iterations for solving \eqref{correction} and \eqref{correction2},
respectively.

From \eqref{inaccrelat} and \eqref{inaccrelat2},
since $\|s_{\perp}\|$ is uncomputable in practice,
we simply replace it by its upper bound one,
which makes $\varepsilon$ as small as possible,
so that the inexact CPF-JDGSVD method is more
reliable to mimic its exact counterpart.

From \eqref{defLK} and the GSVD \eqref{gsvd} of $(A,B)$, the
other eigenvalues of $K$ than $\frac{\sigma^2+1}{\sigma^2-\theta^2}$
are $q_1$-multiple $-\frac{1}{\theta^2}$, $q_2$-multiple $1$
and $\frac{\sigma_i^2+1}{\sigma_i^2-\theta^2}$, $i=2,\dots,q$.
By \eqref{eigK}, they are also the eigenvalues of $K^{\prime}$.
By \eqref{rhoapprox}, we can use
$\sep(\frac{\sigma^2+1}{\sigma^2-\theta^2},K^{\prime})$
to estimate $\sep(\rho,K^{\prime})$:
\begin{eqnarray*}
   \sep(\rho,K^{\prime})&\approx&
   \min\left\{\left|\frac{\sigma^2+1}{\sigma^2-\theta^2}+\frac{1}{\theta^2}\right|,
\left|\frac{\sigma^2+1}{\sigma^2-\theta^2}-1\right|, {\min\limits_{i=2,\ldots,q}
   \left|\frac{\sigma^2+1}{\sigma^2-\theta^2}-
   \frac{\sigma_i^2+1}{\sigma_i^2-\theta^2}\right|}\right\}\\  &\approx& \frac{\sigma^2+1}{|\sigma^2-\theta^2|},
\end{eqnarray*}
where we have used $\theta\approx \sigma$. Since $\theta$ is supposed to approximate $\sigma$, the eigenvalue
$\frac{\sigma^2+1}{\sigma^2-\theta^2}$ is the largest one in magnitude of $K$.
Therefore, it is reasonable to use
$\frac{\sigma^2+1}{|\sigma^2-\theta^2|}$ to estimate $\|K\|$.
Applying these estimates for $\|K\|$ and $\sep(\rho,K^{\prime})$
to \eqref{inaccrelat}, we should terminate the inner iterations of solving the
correction equation \eqref{correction} once
\begin{equation}\label{accsigma}
   \varepsilon\leq 2\tilde\varepsilon
\end{equation}
for a given $\tilde\varepsilon\in [10^{-4},10^{-3}]$; see \eqref{range}.

If $\theta$ is replaced by the fixed target $\tau$, the other eigenvalues
of $K_{\tau}$ than $\frac{\sigma^2+1}{\sigma^2-\tau^2}$ are
$q_1$-multiple $-\frac{1}{\tau^2}$, $q_2$-multiple $1$ and
$\frac{\sigma^2_i+1}{\sigma^2_i-\tau^2}$, $i=2,\dots,q$,
which are also the eigenvalues of $K_{\tau}^{\prime}$.
For \eqref{inaccrelat2}, since the parameter
$\rho_{\tau}=\tilde y^TL_{\tau}\tilde y\approx\frac{\sigma^2+1}{\sigma^2-\tau^2}
\approx\frac{\theta^2+1}{\theta^2-\tau^2}$, we use
$\sep(\frac{\theta^2+1}{\theta^2-\tau^2}, K_{\tau}^{\prime})$
to replace $\sep(\rho,K_{\tau}^{\prime})$ and obtain the estimate
{\small \begin{eqnarray*}
\sep(\frac{\theta^2+1}{\theta^2-\tau^2}, K_{\tau}^{\prime})
&\approx&
   \min\left\{\left|\frac{\theta^2+1}{\theta^2-\tau^2}+\frac{1}{\tau^2}\right|,
\left|\frac{\theta^2+1}{\theta^2-\tau^2}-1\right|, {\min\limits_{i=2,\ldots,q}
   \left|\frac{\theta^2+1}{\theta^2-\tau^2}-
   \frac{\sigma_i^2+1}{\sigma_i^2-\tau^2}\right|}\right\}.\label{estsep}
\end{eqnarray*}}
Observe that the absolute value of the largest eigenvalue in magnitude
of $K_{\tau}$ is
$$
  \max_{i=1,\ldots,q}\left\{\frac{1}{\tau^2},1,
  \frac{\sigma_i^2+1}{|\sigma_i^2-\tau^2|}\right\}.
$$
We use it as an estimate for $\|K_{\tau}\|$. Since
the eigenvalues of $K_{\tau}$ and $K^{\prime}_{\tau}$ are unknown,
we need to further replace the above two a-priori
estimates by exploiting the information available in computations.
Let
$\theta_i,\ i=1,2,\ldots,k$ be the generalized singular values
of $(G,H)$, and suppose that $\theta=\theta_1$
approximates the desired $\sigma$. Then we ultimately estimate $\|K_{\tau}\|$ and
$\sep(\frac{\theta^2+1}{\theta^2-\tau^2}, K_{\tau}^{\prime})$ as follows:
\begin{equation}\label{estktau}
\|K_{\tau}\|\approx  \max_{i=1,\ldots,k}\left\{\frac{1}{\tau^2},1,
  \frac{\theta_i^2+1}{|\theta_i^2-\tau^2|}\right\}
\end{equation}
and
{\small \begin{eqnarray*}
\sep(\frac{\theta^2+1}{\theta^2-\tau^2}, K_{\tau}^{\prime})
&\approx&
   \min\left\{\left|\frac{\theta^2+1}{\theta^2-\tau^2}+\frac{1}{\tau^2}\right|,
\left|\frac{\theta^2+1}{\theta^2-\tau^2}-1\right|, {\min\limits_{i=2,\ldots,k}
   \left|\frac{\theta^2+1}{\theta^2-\tau^2}-
   \frac{\theta_i^2+1}{\theta_i^2-\tau^2}\right|}\right\}.
\end{eqnarray*}}

Define $c_{\tau}$ to be the ratio of the right-hand sides of \eqref{estktau} and
the above relation. 
Then we terminate the
inner iterations of solving the correction equation \eqref{correction2} provided that
$\varepsilon\leq 2c_{\tau}\tilde\varepsilon$.
In practice, in order to guarantee that $\X+span\{\tilde{x}_+\}$ has some
improvement over $\X$, as a safeguard, we propose to take
\begin{equation}\label{acctau}
  \varepsilon\leq \min\{2c_{\tau}\tilde\varepsilon,0.01\}.
\end{equation}

However, $\varepsilon=\frac{\|\tilde t-t\|}{\|t\|}$
is an a-priori error and uncomputable in practice,
which makes us impossible to determine whether or not
$\varepsilon$ becomes smaller than $2\tilde\varepsilon$
or $\min\{2c_{\tau}\tilde\varepsilon,0.01\}$.
Alternatively, denote by
\begin{equation*}
  \|r_{in}\|=
  \tfrac{1}{\|r\|}\left\|-r-(I-\tilde y\tilde x^T)
  (A^TA-\theta^2B^TB)(I-\tilde x \tilde y^T)\tilde t\right\|
\end{equation*}
the relative residual norm of approximate
solution $\tilde t$ of the correction equation \eqref{correction},
and by $\|r_{in,\tau}\|$
the relative residual norm of approximate solution $\tilde t$ of the
correction equation \eqref{correction2}.
Then it is straightforward to justify that
\begin{equation}\label{rinvarep}
  \frac{\varepsilon}{\kappa(M)}
  \leq \|r_{in}\| \leq\kappa(M)\varepsilon
  \quad\mbox{and}\quad
   \frac{\varepsilon}{\kappa(M_{\tau})}
  \leq \|r_{in,\tau}\| \leq\kappa(M_{\tau})\varepsilon,
\end{equation}
where $M$ and $M_{\tau}$ are the matrices
$A^TA-\theta^2B^TB$ and $A^TA-\tau^2B^TB$ restricted to the subspace
$\tilde y^{\perp}$ and map $\tilde y^{\perp}$ to $\tilde x^{\perp}$.
The asymptotic upper bounds for $\kappa(M)$ and $\kappa(M_{\tau})$ are
\eqref{kappaM} and \eqref{kappaMp}, respectively.
Practically, the bounds in \eqref{rinvarep} motivate us
to replace $\varepsilon$ with the corresponding inner relative residual norms
and stop the inner iterations of solving the correction
equations \eqref{correction} and \eqref{correction2}
when
\begin{equation}\label{instopping}
  \|r_{in}\|\leq 2\tilde\varepsilon
  \quad\mbox{and}\quad
  \|r_{in,\tau}\|\leq \min\{2c_{\tau}\tilde\varepsilon,0.01\}
  \end{equation}
  for a given $\tilde\varepsilon$.
When $\kappa(M)$ or $\kappa(M_{\tau})$ is not large,
$\|r_{in}\|$ or $\|r_{in,\tau}\|$ is a reliable replacement of
$\varepsilon$, so that the practical criterion \eqref{instopping}
is robust.

\section{A thick-restart CPF-JDGSVD algorithm with deflation}\label{sec:5}

We discuss extensions and algorithmic developments of the previous basic
CPF-JDGSVD algorithm, which include thick-restart and deflation and
enable us to compute several GSVD components of $(A,B)$.

\subsection{Thick-restart}\label{subsec:4}

As the searching subspaces become large,
the basis matrices $\widetilde U$, $\widetilde V$ and $\widetilde X$
are large. CPF-JDGSVD will be prohibitive due to the excessive
computational complexity. A common approach is to restart the basic algorithm after a
maximum subspace dimension $k_{\max}$ is reached.
We will adapt the thick-restart
technique \cite{stath1998} to our method
for its effectiveness and simplicity in implementations. A main
ingredient is to retain minimal $k_{\min}$ dimensional
left and right searching subspaces
for restart, which are expected to contain some most important
information available on the desired GSVD component at the current cycle.

At the extraction stage, let the GSVD of $(G,H)$ be partitioned as
\begin{equation}\label{partitionr}
   (\Sigma_{G},\Sigma_{H},E,F,D)=\left(
   \bsmallmatrix{\Sigma_{G,1}&\\&\Sigma_{G,2}},
   \bsmallmatrix{\Sigma_{H,1}&\\&\Sigma_{H,2}},
   \bsmallmatrix{E_1,&E_2},
   \bsmallmatrix{F_1,&F_2},
   \bsmallmatrix{D_1,&D_2}
   \right),
\end{equation}
such that  $(\Sigma_{G,1},\Sigma_{H,1},E_1,F_1,D_1)$ is
the partial GSVD  associated with the $k_{\min}$
generalized singular values of $(G,H)$
closest to the target $\tau$, i.e.,
\begin{equation}\label{pgsvdGH}
  GD_1=E_1\Sigma_{G,1}
  \quad\mbox{and}\quad
  HD_1=F_1\Sigma_{H,1}.
\end{equation}

Let the new starting right searching subspace, denoted by $\X_{\rm new}$,
be spanned by the columns of $\widetilde X D_1$.
Then the corresponding starting left searching subspaces,
denoted by $\mathcal{U}_{\rm new}$
and $\V_{\rm new}$, are spanned by the columns of
$A\widetilde X D_1$ and $B\widetilde X D_1$, respectively.
Let $D_1=Q_{r}R_{r}$ be the thin QR factorization of $D_1$.
Then the columns of $\widetilde X_{\rm new}=\widetilde XQ_r$
form an orthonormal basis of $\X_{\rm new}$.
Combining this with \eqref{qrAXBX} and \eqref{pgsvdGH}, we obtain
\begin{eqnarray}
A\widetilde X_{\rm new}
&=&A\widetilde XQ_r=\widetilde U GD_1R_r^{-1}
  = \widetilde UE_1\cdot\Sigma_{G,1}R_r^{-1},\label{qrAXnew}\\
B\widetilde X_{\rm new}
&=&B\widetilde XQ_r=\widetilde V HD_1R_r^{-1}
  = \widetilde VF_1\cdot\Sigma_{H,1}R_r^{-1},\label{qrBXnew}
\end{eqnarray}
where the columns of $\widetilde U E_1$ and $\widetilde V F_1$
are orthonormal, and
$\Sigma_{G,1}R_r^{-1}$ and
$\Sigma_{H,1}R_r^{-1}$ are upper triangular.
Therefore, the right-hand sides of \eqref{qrAXnew} and \eqref{qrBXnew} are
the thin QR factorizations of
$A\widetilde X_{\rm new}$ and $B\widetilde X_{\rm new}$, respectively.
Setting the new matrices $\widetilde U_{\rm new}\!=\!\widetilde U E_1$,
$\widetilde V_{\rm new}\!=\!\widetilde V F_1$ and
$G_{\rm new}\!=\!\Sigma_{G,1}R_r^{-1}$,
$H_{\rm new}\!=\!\Sigma_{H,1}R_r^{-1}$ and rewriting
$\widetilde U$, $\widetilde V$, $\widetilde X$
and $G$, $H$ as $\widetilde U_{\rm new}$,
$\widetilde V_{\rm new}$, $\widetilde X_{\rm new}$ and $G_{\rm new}$,
$H_{\rm new}$, respectively, we then expand the subspaces
in a regular way until they reach the dimension $k_{\max}$ or
the algorithm converges.
In such a way, we have developed a thick-restart CPF-JDGSVD algorithm.

\subsection{Deflation}\label{subsec:5}

Suppose that we are required to compute the $\ell$ GSVD components
$(\sigma_i,u_i,v_i,x_i)$ of $(A,B)$ with $\sigma_i$ closest to
$\tau$, $i=1,\dots,\ell$. By introducing an appropriate deflation technique,
we shall develop a CPF-JDGSVD algorithm for such purpose.
The following result, which is straightforward to justify, forms the
basis of our deflation technique.

\begin{propositionmy}\label{thm:6}
For integer $1 \leq j<\ell$, let the partial GSVD of $(A,B)$
\begin{equation*}
(C_j,S_j,U_j,V_j,X_j)=\left(
\begin{bmatrix}\begin{smallmatrix}
\alpha_1&&\\&\ddots&\\&&\alpha_j
\end{smallmatrix}\end{bmatrix},
\begin{bmatrix}\begin{smallmatrix}
\beta_1&&\\&\ddots&\\&&\beta_j
\end{smallmatrix}\end{bmatrix},
\begin{bmatrix}\begin{smallmatrix}
u_1,&\dots,&u_j
\end{smallmatrix}\end{bmatrix},
\begin{bmatrix}\begin{smallmatrix}
v_1,&\dots,&v_j
\end{smallmatrix}\end{bmatrix},
\begin{bmatrix}\begin{smallmatrix}
x_1,&\dots,&x_j
\end{smallmatrix}\end{bmatrix}\right)
\end{equation*}
be defined by \eqref{gsvd} and \eqref{gsvalue}, and define
$Y_{j}\!=\!(A^TA+B^TB)X_j\!=\!A^TU_{j}C_{j}+B^TV_jS_{j}$.
Then $(\alpha_i,\beta_i,u_i,v_i,x_i)$, $i\!=\!j+1,\ldots,q$ are
the GSVD components of the matrix pair
  \begin{equation}\label{deflatAB}
  (\bm{A}_j,\bm{B}_j):=
  (A(I-X_jY_j^T),B(I-X_jY_j^T)).
\end{equation}
restricted to the orthogonal complement of $span\{Y_j\}$.
\end{propositionmy}

Assume that $(\tilde\alpha_i,\tilde\beta_i,
\tilde u_i,\tilde v_i,\tilde x_i), \ i=1,2,\ldots,j$
are converged approximations to the GSVD components
$(\alpha_i,\beta_i,u_i,v_i,x_i)$ of $(A,B)$ that satisfy the stopping criteria
\begin{equation}\label{converged}
  \|r_{i}\|
  =\|\tilde\beta_{i}A^T\tilde u_{i}
   -\tilde\alpha_{i}B^T\tilde v_{i}\|
  \leq(\tilde\beta_{i}\|A\|_1
  +\tilde\alpha_{i}\|B\|_1)\cdot tol,
  \quad\quad i=1,\dots,j.
\end{equation}
Then
\begin{equation*}
(C_c,S_c,U_c,V_c,X_c)=\left(
\begin{bmatrix}\begin{smallmatrix}
\tilde\alpha_1&&\\&\ddots&\\&&\tilde\alpha_j
\end{smallmatrix}\end{bmatrix},
\begin{bmatrix}\begin{smallmatrix}
\tilde\beta_1&&\\&\ddots&\\&&\tilde\beta_j
\end{smallmatrix}\end{bmatrix},
\begin{bmatrix}\begin{smallmatrix}
\tilde u_1,&\dots,&\tilde u_j
\end{smallmatrix}\end{bmatrix},
\begin{bmatrix}\begin{smallmatrix}
\tilde v_1,&\dots,&\tilde v_j
\end{smallmatrix}\end{bmatrix},
\begin{bmatrix}\begin{smallmatrix}
\tilde x_1,&\dots,&\tilde x_j
\end{smallmatrix}\end{bmatrix}\right)
\end{equation*}
is an approximation to the partial GSVD $(C_j,S_j,U_j,V_j,X_j)$ of
$(A,B)$ that satisfies $AX_c=U_cC_c$, $BX_c=V_cS_c$,
$C_c^2+S_c^2=I_{j}$ and
\begin{equation}\label{conpgsvd}
\|A^TU_cS_c-B^TV_cC_c\|_F=\|[r_1,\dots,r_j]\|_F\leq
\sqrt{j(\|A\|_1^2+\|B\|_1^2)}\cdot tol,
\end{equation}
where the last inequality holds since
$\|r_j\|\leq\sqrt{\|A\|_1^2+\|B\|_1^2}\cdot tol$ from \eqref{converged}.

Denote the matrix
$$
Y_{c}=(A^TA+B^TB)X_c=A^TU_cC_c+B^TV_cS_c.
$$
Then $Y_c^TX_c=X_c^T(A^TA+B^TB)X_c=I$, and $I-X_cY_c^T$ is an
oblique projector onto the orthogonal complement of $span\{Y_c\}$.
Proposition~\ref{thm:6} indicates that, in order to
compute the next GSVD component
$(\alpha,\beta,u,v,x):=(\alpha_{j+1},\beta_{j+1},u_{j+1},v_{j+1},x_{j+1})$
of $(A,B)$, one can apply CPF-JDGSVD to the matrix pair
\begin{equation*}
  (\bm{\widetilde A}_j,\bm{\widetilde B}_j)
  =(A(I- X_{c} Y_c^T),B(I- X_{c} Y_c^T)).
\end{equation*}

If $(\tilde\alpha,\tilde\beta,\tilde u,\tilde v,\tilde x)$ has not yet
converged to $(\alpha,\beta,u,v,x)$,
we set $X_p=[X_c,\tilde x]$ and $Y_p=[Y_c,\tilde y]$ with
$\tilde y$
defined by \eqref{require}.
It is easily seen that the columns of $X_p$ and $Y_p$ are
biorthogonal, i.e., $Y_p^TX_p=I_{j+1}$, and that
$I-X_pY_p^T$ and $I-Y_pX_p^T$ are oblique projectors onto
the orthogonal complements of $span\{Y_p\}$ and $span\{X_p\}$, respectively.
At the expansion stage, if the residual $r$ of $(\tilde\alpha,\tilde\beta,
\tilde u,\tilde v,\tilde x)$ satisfies criterion \eqref{fixtarget},
we switch to solving the correction equation
\begin{equation}\label{multicorrectionsig}
  (I-Y_pX_p^T)(A^TA-\theta^2 B^TB)(I-X_pY_p^T)t=-r
  \quad\mbox{for}\quad t\perp Y_p
\end{equation}
instead of continuing to solve
\begin{equation}\label{multicorrectiontau}
  (I-Y_pX_p^T)(A^TA-\tau^2 B^TB)(I-X_pY_p^T)t=-r
  \quad\mbox{for}\quad t\perp Y_p.
\end{equation}

We remark that the Galerkin condition \eqref{galerkin}
ensures that the current residual $r$ is naturally orthogonal
to the current $\tilde x$ and it is also orthogonal to $X_c$
if $X_{c}=X_{j}$, i.e., the convergence tolerance $tol=0$ in \eqref{conpgsvd}.
On the other hand, for $tol>0$, the residual $r$ is
usually not in the orthogonal complement of $span\{X_c\}$.
For the consistency of the correction equations \eqref{multicorrectionsig} and
\eqref{multicorrectiontau},
we replace the residual $r$ in their right-hand sides
with the projected one
$$
r_{p}=(I-Y_cX_c^T)r
$$
and solve the modified correction equations \eqref{multicorrectionsig} and
\eqref{multicorrectiontau}, respectively.

In the ideal case that $(C_c,S_c,U_c,V_c,X_c)=(C_j,S_j,U_j,V_j,X_j)$,
that is, $tol=0$ in \eqref{converged} and
we have computed the $j$ desired GSVD components exactly,
following the same derivations as those in
Sections \ref{subsec:2}--\ref{subsec:3}, we can directly obtain \eqref{inaccrelat}
and \eqref{inaccrelat2}
for the accuracy $\tilde\varepsilon$ of the expansion vectors and
 for the solution accuracy $\varepsilon$ of the correction
equations \eqref{multicorrectiontau} and
\eqref{multicorrectionsig}
with $K^{\prime}=X_{\perp}^TKX_{\perp}$ and
$K_{\tau}^{\prime}=X_{\perp}^TK_{\tau}X_{\perp}$, respectively,
where the columns of $X_{\perp}$ form an orthonormal
basis of the orthogonal complement $X_{j+1}^{\perp}$ of $span\{X_{j+1}\}$\footnote{
With $j=0$ , the columns of $X_{\perp}$ form an orthonormal basis
of $x_{1}^{\perp}$,
which coincides with the definition of $X_{\perp}$ in Section \ref{sec:3}.}.
For $tol>0$ in \eqref{converged}, a tedious but routine
derivation shows that the new corresponding bounds in \eqref{inaccrelat} and \eqref{inaccrelat2}
are simply the counterparts established for $tol=0$
plus $\OO(tol)$, and we omit details.
Following the same discussions in
Sections \ref{subsec:2}--\ref{subsec:3}, we stop the inner iterations
when the inner relative residual norms $r_{in}$ and $r_{in,\tau}$
satisfy \eqref{instopping} for a given $\tilde\varepsilon$.

With the approximate solution $\tilde t$ of either the modified
correction equation \eqref{multicorrectiontau} or
\eqref{multicorrectionsig}, we orthonormalize
it against $\widetilde X$ to generate the expansion vector
$\widetilde{x}_{+}=\frac{(I-\widetilde X\widetilde X^T)\tilde t}
{\|(I-\widetilde X\widetilde X^T)\tilde t\|}$, and update
$\widetilde X:=[\widetilde X,\tilde x_{+}]$. We then extract a new approximation
to the desired GSVD component with respect to the expanded
$\X=\mathrm{span}\{\widetilde X\}$.
By $\tilde t\perp Y_c$, we have $\widetilde{x}_{+}\perp Y_c$ and
thus $\X\perp Y_c$.

Once $(\tilde\alpha,\tilde\beta,\tilde u,\tilde v,\tilde x)$
has converged in the sense of \eqref{outstopping}, we add it to
the already converged partial GSVD $(C_c,S_c,U_c,V_c,X_c)$ and set $j:=j+1$.
Proceed in such a way until all the $\ell$ desired GSVD components are found.

Motivated by the authors' work \cite{jia2014inner}, we further improve
the above restart approach so as to compute the $(j+1)$-th
GSVD component more efficiently. A key observation is that
the current right searching subspace $\X$ generally provides
reasonably good information on the next desired generalized singular vector.
Therefore, we should fully exploit $\X$ by only purging the newly converged $\tilde x_j$ from the current $\X$
and retaining the resulting reduced subspace, denoted by $\X_{\rm new}$, as a good
initial searching subspace for the next desired GSVD component rather
than from scratch. 
We can do this in the following efficient and numerically stable way.

Let $(\Sigma_{G},\Sigma_{H},E,F,D)$ be the GSVD of
$(G,H)$ partitioned as \eqref{partitionr}, such that
$(\Sigma_{G,1},\Sigma_{H,1},E_1,F_1,D_1)=
(\tilde\alpha,\tilde\beta,e,f,d)$ is the GSVD component corresponding
to the current converged GSVD component
$(\tilde\alpha,\tilde\beta,\tilde u,\tilde v,\tilde x)$ of $(A,B)$.
Since the columns of $D$ are $(G^TG+H^TH)$-orthonormal, we obtain
\begin{eqnarray*}
D_2^T\widetilde X^T\tilde y&=&
  D_2^T\widetilde X^T(A^TA+B^TB)\tilde x
  =D_2^T\widetilde X^T(A^TA+B^TB)\widetilde Xd\\
  &=&D_2^T(G^TG+H^TH)d=\bm{0},
\end{eqnarray*}
that is, the columns of $\widetilde XD_2\perp\tilde y$.
Therefore, $\X_{\rm new}=span\{\widetilde XD_2\}$.
Let $D_2=Q_{d}R_{d}$ be the thin QR factorization of $D_2$.
Then the columns of $\widetilde X_{\rm new}=\widetilde XQ_{d}$ form
an orthonormal basis of $\X_{\rm new}$, so that from \eqref{partitionr}
we can obtain thin QR factorizations:
\begin{eqnarray*}
 A\widetilde X_{\rm new}&=&
 A\widetilde XQ_d=\widetilde UGD_2R_d^{-1}
 =\widetilde UE_2\cdot \Sigma_{G,2}R_d^{-1}, \\
 B\widetilde X_{\rm new}&=&
 B\widetilde XQ_d=\widetilde VHD_2R_d^{-1}
 =\widetilde VF_2\cdot \Sigma_{H,2}R_d^{-1}.
\end{eqnarray*}
Therefore, the columns of
$\widetilde U_{\rm new}=\widetilde UE_2$ and
$\widetilde V_{\rm new}=\widetilde VF_2$ form orthonormal bases of the new
initial left searching subspaces $\mathcal{U}_{\rm new}$ and $\V_{\rm new}$.
We then proceed to expand $\mathcal{X}_{\rm new}$,
$\mathcal{U}_{\rm new}$ and $\V_{\rm new}$ in regular ways.

\subsection{The CPF-JDGSVD algorithm: a pseudocode}\label{code}

The CPF-JDSVD algorithm requires the devices to compute
$A^T\underline u$, $B^T\underline v$ and
$A\underline x$, $B\underline x$ for arbitrary vectors
$\underline u$, $\underline v$ and $\underline x$,
a unit-length starting vector $x_0$ to generate one-dimensional $\X$,
the target $\tau$, the number $\ell$ of the desired GSVD components,
and the convergence tolerance $tol$.
It outputs a converged approximation
$(C_c,S_c,U_c,V_c,X_c)$ to the desired GSVD
$(C_{\ell},S_{\ell},U_{\ell},V_{\ell},X_{\ell})$ of $(A,B)$
associated with the $\ell$ generalized singular values
closest to the target $\tau$ that satisfies $AX_c\!=\!U_cC_c$,
$BX_c\!=\!V_cS_c$, $C_c^2\!+\!S_c^2\!=\!I_{\ell}$ and
\begin{equation}\label{converggsvd}
 \|A^TU_cS_c-B^TV_cC_c\|_F\leq\sqrt{\ell(\|A\|_1^2+\|B\|^2_1)}\cdot tol.
\end{equation}
The other optional parameters are the minimum and maximum dimensions
$k_{\min}$ and $k_{\max}$ of searching subspaces,
the tolerance $fixtol$ used to switch the solution
from  \eqref{correction2} to \eqref{correction},
and the accuracy requirement $\tilde\varepsilon$ on the
expansion vectors in \eqref{instopping}. We set the defaults of
these four parameters as $3$, $30$, $10^{-4}$ and $10^{-3}$, respectively.
Algorithm~\ref{algorithm:1} sketches our thick-restart CPF-JDGSVD algorithm
with deflation.

\begin{algorithm}[htbp]
\caption{CPF-JDGSVD with the target $\tau$.}
\renewcommand{\algorithmicrequire}{\textbf{Input:}}
\renewcommand{\algorithmicensure} {\textbf{Output:}}
\begin{algorithmic}[1]\label{algorithm:1}
\STATE{Initialization:\ Set $k=1$, $k_c=0$, $C_c=[\ \ ], S_c=[\ \ ], U_c=[\ \ ],
V_c=[\ \ ],X_c=[\ \ ]$ and $Y_c=[\ \ ]$.
Let $\widetilde U=[\ \ ]$, $\widetilde V=[\ \ ]$,
$\widetilde X=[\ \ ]$ and $x_{+}=x_0$.}

\WHILE{$k\geq0$}

\STATE{\label{step:3}Set $\widetilde X=[\widetilde X, x_{+}]$,
and update the QR factorizations
$A\widetilde X=\widetilde UG$ and
$B\widetilde X=\widetilde VH$.}

\STATE{\label{step:4}Compute the GSVD of $(G,H)$, label
the generalized singular values in increasing order according to their distances
from the target $\tau$,
and pick up $(\tilde\alpha,\tilde\beta,e,f,d)$ with
the generalized singular value
$\theta=\frac{\tilde\alpha}{\tilde\beta}$
closest to $\tau$.}

\STATE{\label{step:5}Compute the approximate generalized singular vectors
$\tilde u=\widetilde Ue$,
$\tilde v=\widetilde Vf$,
$\tilde x=\widetilde Xd$,
the residual $r=\tilde \beta A^T\tilde u
-\tilde\alpha B^T\tilde v$, and
$\tilde y=\tilde\alpha A^T\tilde u+\tilde\beta B^T\tilde v$.}

\STATE{
\textbf{if} $\|r\|\leq(\tilde\beta\|A\|_1+\tilde\alpha\|B\|_1)\cdot tol$
\textbf{then} (\romannumeral1) set $k_c=k_c+1$ and update
$C_c=\bsmallmatrix{C_c&\\&\tilde\alpha}$,
$S_c=\bsmallmatrix{S_c&\\&\tilde\beta}$,
$U_c=[U_c,\tilde u]$, $V_c=[V_c,\tilde v]$,
$X_c=[X_c,\tilde x]$ and $Y_c=[Y_c, \tilde y]$;
(\romannumeral2) \textbf{if} $k_c=\ell$  \textbf{then}
return
$(C_c,S_c,U_c,V_c,X_c)$ and stop;
(\romannumeral3) set $k=k-1$,  
purge $\tilde x$ from the current $\X$,
obtain the reduced subspaces, and go to step 3.
}

\STATE{Set $X_{p}=[X_c,\tilde x]$ and $Y_{p}=[Y_c,\tilde y]$.
\textbf{if} $\|r\|\leq(\tilde\beta\|A\|_1+
\tilde\alpha\|B\|_1)\cdot fixtol$ \textbf{then} solve the correction equation
\begin{equation*}
  (I-Y_pX_p^T)(A^TA-\theta^2 B^TB)(I-X_pY_p^T)t=-(I-Y_cX_c^T)r
  \quad\mbox{with}\quad t\perp Y_p
\end{equation*}
for an approximate solution $\tilde t$ by requiring that
the inner relative residual norm $\|r_{in}\|\leq 2\tilde\varepsilon$;
\textbf{else} solve the correction equation
\begin{equation*}
  (I-Y_pX_p^T)(A^TA-\tau^2 B^TB)(I-X_pY_p^T)t=-(I-Y_cX_c^T)r
  \quad\mbox{with}\quad t\perp Y_p
\end{equation*}
for an approximate solution $\tilde t$ by requiring that the inner
relative residual norm
$\|r_{in,\tau}\|\leq\min\{2c_{\tau},0.01\}\tilde\varepsilon$.
}

\STATE{\textbf{if} $k=k_{\max}$ \textbf{then} set $k=k_{\min}$,
and perform a thick-restart.}

\STATE{Orthonormalize $\tilde t$ against $\widetilde X$
to get the expansion vector $\tilde x_{+}$ and set $k=k+1$.}
\ENDWHILE
\end{algorithmic}
\end{algorithm}

\section{Numerical examples}\label{sec:7}

We now report numerical results on several problems
to illustrate the efficiency of Algorithm~{\ref{algorithm:1}}.
All the experiments were performed on an Intel (R) Core (TM) i7-7700
CPU 3.60 GHz with the main memory 8 GB and 4 cores using the
MATLAB R2020b with the machine precision
$\epsilon_{\rm mach} =2.22\times10^{-16}$
under the Windows 10 64-bit system.

Tables~\ref{table0} lists the test matrix pairs with some of their
basic properties, where we use sparse matrices from the SuiteSparse
Matrix Collection \cite{davis2011university} or their transposes,
denoted by the matrix names with the superscript $T$,
as our test matrices $A$, so as to ensure $m\geq n$,
and the matrices $B$ are the $n\times n$ tridiagonal Toeplitz matrix
$B_0$ with $3$ and $1$ being the main and off diagonal elements
and the $(n-1)\times n$ scaled discrete approximation
$B_1$ of the first order derivative operator of
dimension one \cite{hansen1998rank}, i.e.,
\begin{equation}\label{defB}
B_0=\begin{bmatrix}3&1&&\\1&\ddots&\ddots&
\\&\ddots&\ddots&1 \\[0.2em]&&1&3\end{bmatrix}
\qquad\mbox{and}\qquad
B_1=\begin{bmatrix}1&-1&&\\&\ddots&\ddots&
\\&&1&-1 \end{bmatrix}.
\end{equation}
In order to verify the reliability and behavior of CPF-JDGSVD,
we have used the MATLAB built-in functions {\sf gsvd} and {\sf eig}
to compute the GSVD of the first five and last six test problems
for $m\approx n$ roughly and $m\gg n$, respectively, with
{\sf eig} applied to $(A^TA,B^TB)$.
For the three large matrix pairs $(A,B)=(\mathrm{tmgpc1},B_0)$,
$(\mathrm{ wstn\_1}^T,B_0)$ and $(\mathrm{degme}^T,B_0)$,
we have applied the MATLAB built-in function {\sf svds} to the
stacked matrix $\bsmallmatrix{A\\B}$ to compute its largest and
smallest singular values and obtained $\kappa\left(\bsmallmatrix{A\\B}\right)$.
We have used the MATLAB built-in function {\sf eigs} to
$(A^TA,B^TB)$ to compute the largest and smallest generalized singular
values of $(A,B)$. Particularly, the smallest generalized singular values
$\sigma_{\min}=0$ of $(\mathrm{tmgpc1},B_0)$ and
$(\mathrm{stat96v5}^T,B_1)$
since $\mathrm{tmgpc1}$ and $\mathrm{stat96v5}^T$
are known to be rank deficient \cite{davis2011university}.

\begin{table}[tbhp]
\caption{Properties of the test matrix pairs $(A,B)$, where $B_0$
and $B_1$ are defined by \eqref{defB},
`$\mathrm{tmgpc1}$', `$\mathrm{wstn\_1}$' and `flower54' are
abbreviations of `tomographic1', `waston\_1' and `flower\_5\_4',
respectively, $nnz$ is the total number of nonzero entries in
$A$ and $B$, and $\sigma_{\max}$ and $\sigma_{\min}$ are the largest and smallest
generalized singular values of $(A,B)$, respectively.}\label{table0}
\begin{center}
\begin{tabular}{ccccccccc} \toprule
{$A$}&{$B$}&$m$&$n$&$p$&$nnz$&$\kappa(\bsmallmatrix{A\\B})$
&$\sigma_{\max}$&$\sigma_{\min}$\\ \midrule
$\mathrm{r05}^T$      &$B_0$     &9690    &5190   &5190   &119713
                      &13.2    &17.4    &3.10e-2     \\
$\mathrm{deter4}^T$   &$B_0$     &9133    &3235   &3235   &28934
                      &7.07    &8.56    &5.62e-3    \\
$\mathrm{lp\_bnl2}^T$ &$B_0$     &4486    &2324   &2324   &21966
                      &1.93e+2 &1.10e+2 &1.20e-2  \\
$\mathrm{large}^T$    &$B_0$     &8617    &4282   &4282   &33479
                      &3.53e+3 &2.40e+3 &2.25e-3   \\
$\mathrm{gemat1}^T$   &$B_0$     &10595   &4929   &4929   &61376
                      &1.82e+4 &1.21e+4 &5.97e-5   \\
$\mathrm{tmgpc1}$     &$B_0$     &73159   &59498  &59498  &825987
                      &7.94    &2.78    &0 \\
$\mathrm{wstn\_1}^T$  &$B_0$     &\!\!386992\!\!  &\!\!201155\!\!
                      &\!\!201155\!\!   &\!\!1658556\!\!
                      &15.1    &11.1&5.19e-3 \\
$\mathrm{degme}^T$    &$B_0$     &\!\!659415\!\!  &\!\!185501\!\!
                      &\!\!185501\!\!   &\!\!8684029\!\!
                      &2.04e+2 &9.24e+2 &1.24  \\
$\mathrm{slptsk}^T$   &$B_1$   &3347    &2861   &2860   &78185
                      &3.29e+2 &1.16e+4 &1.81e-1   \\
$\mathrm{rosen10}^T$  &$B_1$   &6152    &2056   &2055   &68302
                      &1.57e+2 &2.02e+4 &1.24  \\
$\mathrm{flower54}^T$ &$B_{1}$   &14721   &5226   &5225   &54392
                      &4.52    &7.73e+3 &2.42e-1  \\
$\mathrm{l30}^T$      &$B_1$   &16281   &2701   &2700   &57470
                      &7.76    &2.53e+3 &2.22e-3 \\
$\mathrm{cq5}^T$      &$B_1$   &11748   &5048   &5047   &61665
                      &8.91e+3 &7.34e+4 &3.78e-2 \\
$\mathrm{stat96v5}^T$ &$B_1$   &75779   &2307   &2306   &238533
                      &31.7    &5.55e+3 &0 \\
\bottomrule
\end{tabular}
\end{center}
\end{table}

For each matrix pair $(A,B)$ with a given target $\tau$,
we compute the GSVD components of $(A,B)$ corresponding
to the $\ell$ generalized singular values closest to $\tau$,
where $\ell=1,5$, and $9$.
We take the initial vector $x_0$ for the problems with the full column rank $B$
to be $\frac{1}{\sqrt{n}} [1,1,\ldots,1]^T$,
and $x_0$ for the problems with $B$ rank deficient to be the
unit-length vector whose primitive $i$-th
element is $(i\mod 4)$, $i=1,\dots,n$.
An approximate GSVD component
$(\tilde{\alpha},\tilde{\beta},\tilde{u},
\tilde{v},\tilde{x})$
obtained by the CPF-JDGSVD algorithm is claimed to have converged
if its residual norm satisfies \eqref{converged} with $tol=10^{-10}$.
For the inner iterations, we use the unpreconditioned MINRES to solve the
correction equation
\eqref{multicorrectionsig} or \eqref{multicorrectiontau},
where the code {\sf minres} is from MATLAB R2020b.
We always take the initial approximate solutions to be zero vectors
and stop the inner iterations when the
stopping criterion \eqref{instopping} is fulfilled for a
fixed $\widetilde\varepsilon$. Unless specified otherwise, we take the parameters
in CPF-JDGSVD to be the defaults in Section~\ref{code}.


As a comparison, we also compute the desired GSVD components of $(A,B)$ with the full column rank $B$ using the JDGSVD algorithm \cite{hochstenbach2009jacobi}
with the parameters as far as possible the same as in the thick-restart CPF-JDGSVD.
We always take $u_0=\frac{1}{\sqrt{m}}[1,\dots,1]^T$
as the initial left vector for JDGSVD, which works on the generalized eigenvalue problem of
$\left(\bsmallmatrix{&A\\A^T&},\bsmallmatrix{I&\\&B^TB}\right)$
for the full column rank $B$.
At each step, the JDGSVD algorithm computes an approximation
$(\vartheta, \hat u, \hat w)$ to the desired triplet $(\sigma,u,w:=x/\beta)$.
In our implementations, an approximate $(\vartheta, \hat u, \hat w)$ is claimed to
have converged if the residual
$$
r=\begin{bmatrix}A\hat w-\vartheta\hat u\\
           A^T\hat u-\vartheta B^TB\hat w\end{bmatrix}=
           \left(\begin{bmatrix}&A\\A^T&\end{bmatrix}-\vartheta \begin{bmatrix}I&\\&B^TB \end{bmatrix}\right)
           \begin{bmatrix}\hat u\\ \hat w\end{bmatrix}
$$
of the approximate generalized eigenpair $(\vartheta,(\hat{u}^T,\hat{w}^T)^T)$
satisfies
\begin{equation}\label{resJDGSVD}
  relres:=\frac{\|r\|}{(\|A\|_1+\vartheta\|B^TB\|_1)\sqrt{1+\|\hat w\|^2}}\leq tol
\end{equation}
with $tol$ the same as in \eqref{outstopping}.
For the correction equations (cf. the equation after equation (13)
in \cite{hochstenbach2009jacobi}.),
we take $\theta$ to be $\tau$ in the initial steps and then
switch to the approximate generalized singular value $\vartheta$ when
the relative residual norm $relres$ is smaller than the same switching
tolerance $fixtol$ in CPF-JDGSVD. We always take zero vector as the initial guess and
use {\sf minres} to solve the symmetric correction
equations until the relative residual norm of the inner iterations is
smaller than $2\widetilde{\varepsilon}$, where
$\widetilde{\varepsilon}$ is the same as in \eqref{accsigma} and \eqref{acctau}.
To make a fair comparison, we have introduced the thick-restart and
deflation technique similar to those
described in Section~\ref{sec:5} into JDGSVD for computing $\ell$ GSVD
components of $(A,B)$. Once a triplet $(\vartheta,\hat u,\hat w)$ has converged,
the corresponding converged approximate GSVD component is recovered by
\begin{equation*}
  (\hat\alpha,\hat\beta,\hat u,\hat v,\hat x)=
  (\frac{\vartheta}{\sqrt{1+\vartheta^2}},\frac{1}
  {\sqrt{1+\vartheta^2}},\hat u,\frac{1}{\|B\hat w\|}B\hat w,
  \frac{1}{\sqrt{\|A\hat w\|^2+\|B\hat w\|^2}}\hat w).
\end{equation*}
As an approximation to the desired GSVD component $(\alpha,\beta,u,v,x)$ of
$(A,B)$, for the {\em original} GSVD problem, the associated true relative residual
norm of $(\hat\alpha,\hat\beta,\hat u,\hat v,\hat x)$ is
\begin{equation}\label{truerelres}
  relres_{t}=
  \frac{\|A\hat x-\hat\alpha\hat u\|}{\|A\|_1\|\hat x\|+\hat\alpha}+
  \frac{\|B\hat x-\hat\beta\hat v\|}{\|B\|_1\|\hat x\|+\hat\beta}+
  \frac{\|\hat\beta A^T\hat u-\hat\alpha B^T\hat v\|}
  {\hat\beta\|A\|_1+\hat\alpha\|B\|_1}.
\end{equation}
For the CPF-JDGSVD algorithm, the first two terms in the right-hand side vanish.

In all the tables, we denote by $I_{out}$ and $I_{in}$ the total numbers of
outer and inner iterations, respectively,
and by $T_{cpu}$ the CPU time in seconds counted by the MATLAB
built-in commands {\sf tic} and {\sf toc}.

\begin{exper}\label{exp1}
We compute the GSVD components of $(A,B)=(\mathrm{r05}^T,B_0)$
with $\ell=1,5,9$ corresponding to the generalized singular values closest to
$\tau=4$ using the CPF-JDGSVD algorithm with $fixtol=+\infty, 10^{-2},10^{-4}$
and $0$, respectively.
Here $fixtol=+\infty$ or $0$ is a virtual value and means that
we always solve the modified correction equation
\eqref{multicorrectionsig} or \eqref{multicorrectiontau} only.
The desired generalized singular values of $(A,B)$ are
clustered interior ones.
\end{exper}

\begin{table}[tbhp]
\caption{$(A,B)=(\mathrm{r05}^T,B_0)$ with $\tau=4$.}\label{table1}
\begin{center}
\begin{tabular}{cccccccccc} \toprule
\multirow{2}{*}{$fixtol$}
&\multicolumn{3}{c}{$\ell=1$}
&\multicolumn{3}{c}{$\ell=5$}
&\multicolumn{3}{c}{$\ell=9$}\\
\cmidrule(lr){2-4} \cmidrule(lr){5-7} \cmidrule(lr){8-10}
&$I_{out}$&$I_{in}$&$T_{cpu}$
&$I_{out}$&$I_{in}$&$T_{cpu}$
&$I_{out}$&$I_{in}$&$T_{cpu}$\\ \midrule
$+\infty$  &207&1044081&363 &319&1613267&565  &341&1707652&605\\
$10^{-2}$  &7&1070&0.38     &28&4889&1.88     &50&12777&5.20 \\
$10^{-4}$  &10&1212&0.46    &36&5116&2.07     &71&19217&8.05 \\
$0$        &13&1587&0.59    &65&6933&2.86     &337&30534&14.8 \\
\bottomrule
\end{tabular}
\end{center}
\end{table}

Table~\ref{table1} reports the results. Clearly, for the three $\ell$,
CPF-JDGSVD with $fixtol=+\infty$ uses much more outer and inner iterations and much
more CPU time to converge than it does for the other three $fixtol$.
What is worse, none of the converged generalized singular values
$\sigma_{c,i}\in[1.93,2.17]$ is a desired one since the $\ell$
desired $\sigma_{i}\in[3.68,4.47]$, meaning that the algorithm misconverges.
As a matter of fact, CPF-JDGSVD  with $fixtol=10^{-2}$ has the
same issue: the first converged generalized singular value
$\sigma_{c,1}\approx4.12$ is not the closest to $\tau$ but the
second converged $\sigma_{c,2}\approx3.92$ is.
In contrast, CPF-JDGSVD with $fixtol=0$ converges correctly thought it
uses more outer iterations than CPF-JDGSVD with
$fixtol=10^{-2}$.
CPF-JDGSVD with $10^{-4}$ works reliably and uses much fewer outer iterations
than it does with $fixtol=0$.
The results indicate that in order to make the algorithm reliable and
efficient, one should take a relatively small $fixtol$.

As we observe from the table, for $\ell=9$, CPF-JDGSVD with $fixtol=10^{-4}$ converges
significantly faster than it does with $fixtol=0$, and the total
inner iterations are substantially reduced as well.
Obviously, with an inappropriately larger or smaller $fixtol$,
CPF-JDGSVD may compute wrong GSVD components or converge
very slowly. We have also observed the same phenomena on
other test matrix pairs.
A good choice of $fixtol$ must guarantee the reliability of the
computed GSVD components and, meanwhile, should reduce the total
computational costs as much as possible. Such a choice is obviously problem
dependent. Nonetheless, we have found from the experiments on the other
problems that, for the reliability and efficiency
of CPF-JDGSVD, $fixtol=10^{-4}$ is a good choice and
is used as a default.

\begin{exper}\label{exp2}
We compute the GSVD components of $(A,B)=(\mathrm{deter4}^T,B_0)$ with $\ell=1,5,9$
corresponding to the clustered interior generalized singular
values closest to $\tau=0.08$ using the CPF-JDGSVD
algorithm with $\widetilde\varepsilon=10^{-3},10^{-4},10^{-15}$
in \eqref{instopping}, where $\widetilde\varepsilon=10^{-15}$ means that
all the correction equations have been numerically solved exactly
in finite precision arithmetic.
For the experimental purpose, we have also used the so-called
``exact'' CPF-JDGSVD algorithm to compute the desired GSVD components, where
``exact'' means, as indicated by \eqref{expressiony} and \eqref{defnu},
that the correction equations
\eqref{multicorrectionsig} and \eqref{multicorrectiontau}
are solved by the LU factorizations of  $L=A^TA-\theta^2 B^TB$
and $L_{\tau}=A^TA-\tau^2 B^TB$, respectively.
\end{exper}

\begin{table}[tbhp]
\caption{$(A,B)=(\mathrm{deter4}^T,B_0)$ with $\tau=0.08$. 
}\label{table2}
\begin{center}
\begin{tabular}{cccccccccc} \toprule
\multirow{2}{*}{$\tilde\varepsilon$}
&\multicolumn{3}{c}{$\ell=1$}
&\multicolumn{3}{c}{$\ell=5$}
&\multicolumn{3}{c}{$\ell=9$}\\
\cmidrule(lr){2-4} \cmidrule(lr){5-7} \cmidrule(lr){8-10}
&$I_{out}$&$I_{in}$&$T_{cpu}$
&$I_{out}$&$I_{in}$&$T_{cpu}$
&$I_{out}$&$I_{in}$&$T_{cpu}$\\ \midrule
$10^{-3}$  &24&2082&0.29  &38&8790&1.02  &49&13711&2.03\\
$10^{-4}$  &24&2586&0.27  &33&8494&0.93  &47&13748&2.03\\
$10^{-15}$ &23&7192&0.72  &37&19199&2.00 &46&25972&3.36 \\
exact      &23&-&2.04     &37&-&3.32     &46&-&4.22 \\ \bottomrule
\end{tabular}
\end{center}
\end{table}

\begin{figure}[tbhp]
\centering
\label{fig1b}\includegraphics[width=0.85\textwidth]{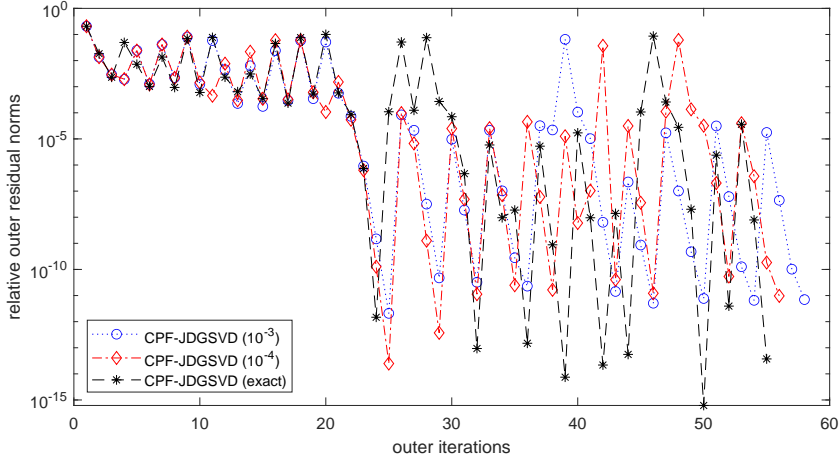}
\caption{Computing the GSVD components of $(A,B)=(\mathrm{deter4}^T,B_0)$
corresponding to the nine generalized singular values closest to $\tau=0.08$.}
\label{fig1}
\end{figure}

Tabel~\ref{table2} reports the results, and Figure~\ref{fig1} depicts
the convergence curves of CPF-JDGSVD with $
\widetilde\varepsilon=10^{-3}$, $10^{-4}$
and the ``exact'' CPF-JDGSVD for computing nine GSVD components of $(A,B)$,
where the GSVD components are computed one by one and the
convergence curve, therefore, has nine stages for
each $\widetilde\varepsilon$. We observe from Table~\ref{table2} and
Figure~\ref{fig1} that,
regarding the outer iterations, for $\ell=1,5$ and $9$,
CPF-JDGSVD with $\widetilde\varepsilon=10^{-3}$
and $10^{-4}$ behaves very much like its exact counterpart.
Furthermore, we have found that, compared with the iterative exact CPF-JDGSVD,
i.e., $\widetilde\varepsilon\!=\!10^{-15}$,
the inexact CPF-JDGSVD algorithm costs only less than $53\%$
of total inner iterations or less than $61\%$ of total CPU time
to compute the desired GSVD components.
Clearly, a smaller $\widetilde\varepsilon$ is unnecessary since it cannot
reduce outer iterations and instead increases the total cost substantially.
Therefore, in the sequel, we
adopt the default $\widetilde\varepsilon\!=\!10^{-3}$ in CPF-JDGSVD and
JDGSVD.

\begin{exper}\label{exp3}
 We compute the GSVD components of some other problems in Table~\ref{table0}.
 We write the matrix pairs $(A_a,B_a)\!=\!(\mathrm{lp\_bnl2}^T,B_0)$,
$(A_b,B_b)\!=\!(\mathrm{large}^T,B_0)$ and
$(A_c,B_c)\!=\!(\mathrm{gemat1}^T,B_0)$
with the targets $\tau_a=20$, $\tau_b=5$ and $\tau_c=12$, respectively.
The desired GSVD components are all clustered interior ones.
We also test the large scale matrix pairs
$(A_d,B_d)\!=\!(\mathrm{tmgpc1},B_0)$, $(A_e,B_e)\!=\!(\mathrm{wstn\_1}^T,B_0)$
and $(A_f,B_f)\!=\!(\mathrm{degme}^T,B_0)$
with $\tau_d=2.7$, $\tau_e=7$ and $\tau_f=1.3$, respectively.
The desired GSVD components correspond to the
largest, interior and smallest ones of
$(A_d,B_d)$, $(A_e,B_e)$ and $(A_f,B_f)$, respectively.
\end{exper}

\begin{table}[tbhp]
\caption{Results of CPF-JDGSVD on some of the problems in Table \ref{table0}.}\label{table3}
\begin{center}
\begin{tabular}{cccccccccc} \toprule
\multirow{2}{*}{$A$}
&\multicolumn{3}{c}{$\ell=1$}
&\multicolumn{3}{c}{$\ell=5$}
&\multicolumn{3}{c}{$\ell=9$}\\
\cmidrule(lr){2-4} \cmidrule(lr){5-7} \cmidrule(lr){8-10}
&$I_{out}$&$I_{in}$&$T_{cpu}$
&$I_{out}$&$I_{in}$&$T_{cpu}$
&$I_{out}$&$I_{in}$&$T_{cpu}$\\ \midrule
$\mathrm{lp\_bnl2}^T$
&11&497&0.09
&38&1841&0.28
&68&\!\!3752\!\!&0.54\\
$\mathrm{large}^T$
&7&\!\!9528\!\!  &1.71
&47&\!\!76190\!\!&13.7
&61&\!\!132107\!\!&25.6 \\
$\mathrm{gemat1}^T$
&8&\!\!6900\!\!&1.90
&23&\!\!22257\!\!&6.71
&40&\!\!33336\!\!&10.7 \\
$\mathrm{tmgpc1}$
&8&403&1.82
&27&1234&6.31
&54&2752&15.0 \\
$\mathrm{wstn\_1}^T$
&13 &\!\!\!\!35982\ \ \!\!\!\!  &\!\!\!\!5.29e+2\!\!\!\!
&32 &\!\!\!\!106212\ \ \!\!\!\! &\!\!\!\!1.90e+3\!\!\!\!
&51 &\!\!\!\!167817\ \ \!\!\!\! &\!\!\!\!3.46e+3\!\!\!\! \\
$\mathrm{degme}^T$
&\!\!8\!\!&92&5.16
&\!\!112\!\!&1399&83.3
&\!\!261\!\!&3347&\!\!2.01e+2\!\!\\ \bottomrule
\end{tabular}
\end{center}
\end{table}

\begin{table}[tbhp]
\caption{Results of JDGSVD on some of the problems in Table \ref{table0}.}\label{table4}
\begin{center}
\begin{tabular}{ccccccccccl} \toprule
\multirow{2}{*}{$A$}
&\multicolumn{3}{c}{$\ell=1$}
&\multicolumn{3}{c}{$\ell=5$}
&\multicolumn{4}{c}{$\ell=9$}\\
\cmidrule(lr){2-4} \cmidrule(lr){5-7} \cmidrule(lr){8-11}
&\!\!$I_{out}$\!\!&$I_{in}$&$T_{cpu}$
&\!\!$I_{out}$\!\!&$I_{in}$&$T_{cpu}$
&\!\!$I_{out}$\!\!&$I_{in}$&$T_{cpu}$ & $Relres_t$\\ \midrule
\!\!$\mathrm{lp\_bnl2}^T$\!\!
&\!\!10\!\!&\!\!918\!\!&0.20
&\!\!37\!\!&\!\!3023\!\!&0.78
&\!\!64\!\!&\!\!5640\!\!&1.51
&2.19e-10\!\! \\
\!\!$\mathrm{large}^T$\!\!
&\!\!7\!\!&\!\!9786\!\!&2.80
&\!\!26\!\!&\!\!71127\!\!&24.4
&\!\!43\!\!&\!\!137692\!\!&49.7
&2.08e-10 \\
\!\!$\!\!\mathrm{gemat1}^T$\!\!
&\!\!13\!\!&\!\!10846\!\!&5.59
&\!\!31\!\!&\!\!24986\!\!&13.6
&\!\!48\!\!&\!\!36553\!\!&20.4
&1.75e-10\!\! \\
\!\!$\mathrm{tmgpc1}$\!\!
&\!\!8\!\!&\!\!807\!\!&6.29
&\!\!24\!\!&\!\!1987\!\!&17.2
&\!\!42\!\!&\!\!4282\!\!&40.4
&1.25e-9\!\! \\
\!\!$\mathrm{wstn\_1}^T$\!\!
&\!\!10\!\!&\!\!\!\!51244\!\!\!\!& \!\!2.17e+3\!\!\!\!
&\!\!23\!\!&\!\!\!\!131417\!\!\!\!&\!6.40e+3\!\!
&\!\!41\!\!&\!\!\!\!206423\!\!\!\!&\!\!1.11e+4\!\!\!\!
&4.25e-10\!\! \\
\!\!$\mathrm{degme}^T$\!\!
&\!\!10\!\!&\!\!582\!\!&\!\!40.4\!\!
&\!\!\!\!353\!\!\!\!&\!\!\!\!20673\!\!\!\!&\!1.75e+3\!\!
&\!\!\!\!998\!\!\!\!&\!\!\!\!64200\!\!\!\!&\!\!\!\!6.79e+3\!\!\!\!
&1.22e-9\!\! \\
\bottomrule
\end{tabular}
\end{center}
\end{table}


For these six test problems, we have observed very similar phenomena
to those for the previous two examples.
Table~\ref{table3} reports the results, where the results on $(A_e,B_e)$
are obtained by CPF-JDGSVD with
the fixed correction equation \eqref{multicorrectiontau} solved,
since for this problem we have noticed that CFP-JDGSVD with $fixtol=10^{-4}$
used too many inner iterations but comparable outer iterations.
Table~\ref{table3} indicates that CPF-JDGSVD worked efficiently
for computing both the interior and the extreme GSVD components of the test
matrix pairs.
Particularly, we have seen that the outer iterations for $\ell=5$ and $9$
are only slightly more than those for $\ell=1$, confirming the
effectiveness of the restarting scheme proposed in Section 4.2,
where the reduced $\widetilde{\mathcal{X}}_{\rm new}$'s of
purging the converged right generalized singular vectors from
the current subspaces indeed retain
rich information on the next desired right generalized singular
vectors.

For these six problems, all the $B$ are well conditioned
with $\kappa(B)\approx5$.
We have also applied the JDGSVD algorithm \cite{hochstenbach2009jacobi}
to these problems
with all the parameters same as in the CPF-JDGSVD algorithm.
Table~\ref{table4} displays the results, where $Relres_t$ is
the relative residual norm whose entries are
the $\ell$ $relres_t$ of converged GSVD components defined by \eqref{truerelres}.
We see from Table~\ref{table4} that for all the six problems,
the relative residual norms of the converged approximate GSVD
components computed by JDGSVD are very comparable to the
stopping tolerance $10^{-10}$, as is expected since matrices $B$ are very
well conditioned.
Comparing Table~\ref{table3} with Table~\ref{table4}, we can see that
CPF-JDGSVD uses very comparable outer iterations as JDGSVD for the
first five matrix pairs but it is at least three times as fast as
JDGSVD for the last problem with $\ell=5,9$.
Regarding the overall efficiency, CPF-JDGSVD uses
fewer inner iterations or less than $52\%$ of CPU time to compute
the desired nine GSVD components of $(A_b,B_b)$, $(A_c,B_c)$ and $(A_e,B_e)$.
It reduces more than $33\%$ of inner iterations and more than
$63\%$ of CPU time to compute all the desired GSVD components of $(A_a,B_a)$ and $(A_b,B_b)$. For $(A_f,B_f)$,
CPF-JDGSVD significantly outperforms JDGSVD by using
$5\%$ of inner iterations and $3\%$ of CPU time to converge.
Therefore, for the matrix pairs with the full column rank and well-conditioned $B$,
CPF-JDGSVD is more efficient than or at
least competitive with JDGSVD.

\begin{exper}\label{exp4}
We use CPF-JDGSVD and JDGSVD to compute nine GSVD components 
of  $(A,B)=(\mathrm{blckhole},B_1^T)$ corresponding to
the generalized singular values closest to $\tau=1000$,
where $\mathrm{blckhole}$ is a $2132\times 2132$ sparse matrix 
from \cite{davis2011university} and $B=B_1^T$ ensures that $B$ has full column rank.
The desired generalized singular values are the largest
ones of $(A,B)$ and are well separated {\color{red} from} each other.
\end{exper}

\begin{figure}[tbhp]
\centering
\subfloat[]{\label{fig2a}
\includegraphics[width=6cm,height=4cm]{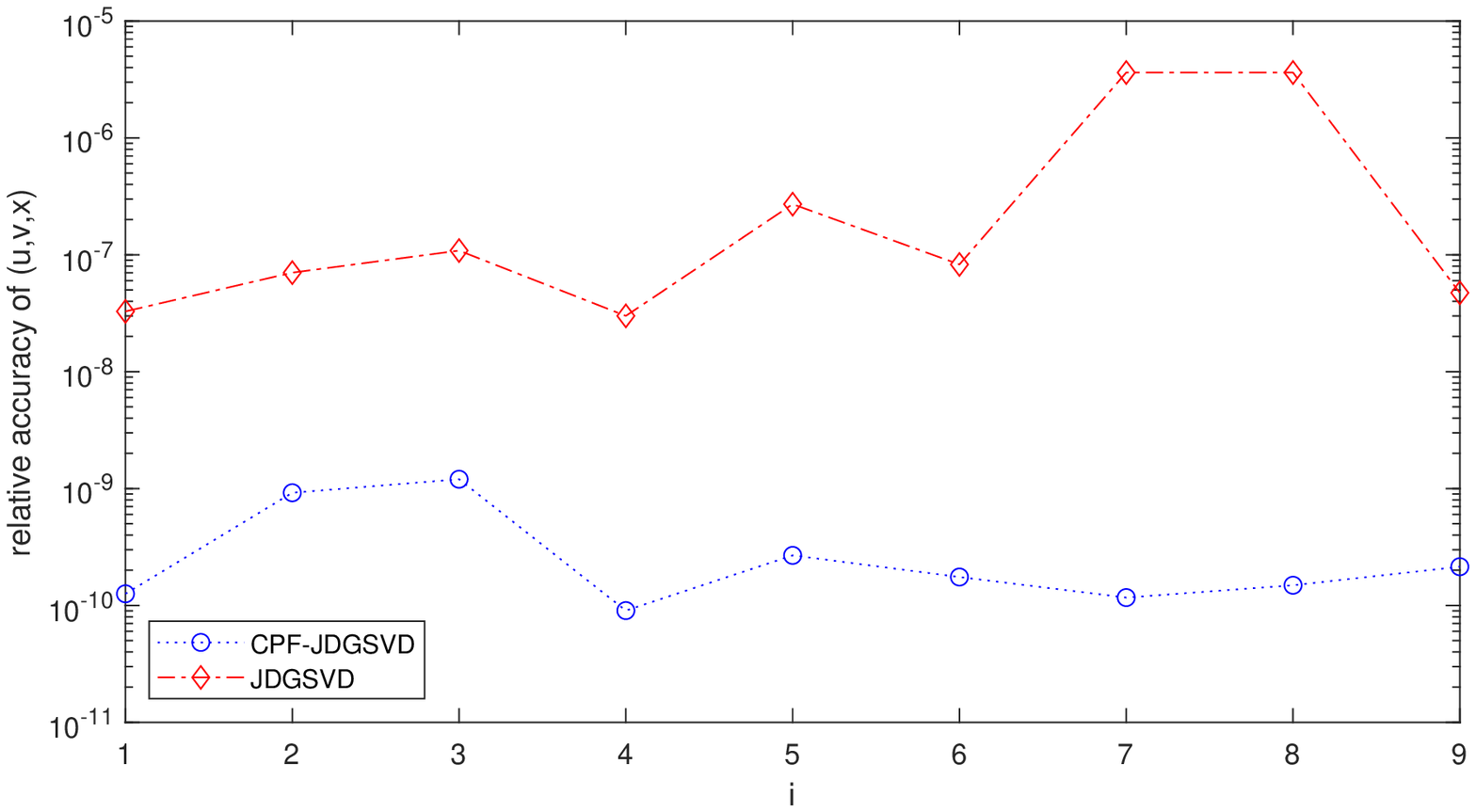}}
\hfill
\subfloat[]{\label{fig2b}
\includegraphics[width=6cm,height=4cm]{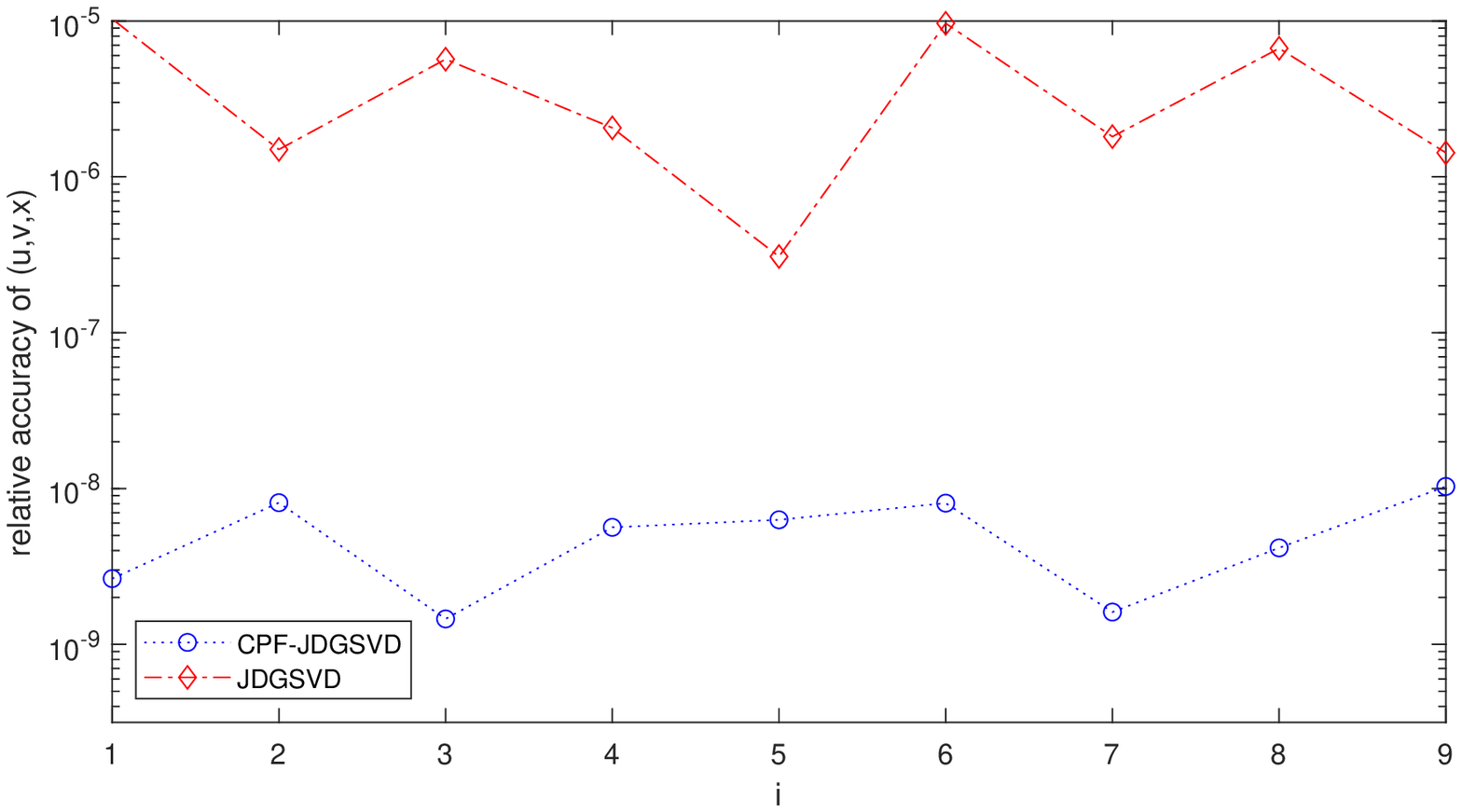}}
\caption{(a) and (b) Accuracy of the converged generalized singular vectors of
$(A,B)=(\mathrm{blckhole},B_1^T)$ with $\tau=10^3$ and
$(A,B)=(\mathrm{plddb}^T,B_2^T)$ with $\tau=70$ computed by CPF-JDGSVD and JDGSVD, respectively.}
\label{fig2}
\end{figure}

\begin{table}[tbhp]
\caption{Results on $A=\mathrm{blckhole}$ and $B=B_1^T$.}\label{table5}
\begin{center}
\begin{tabular}{ccccccccc} \toprule
\multirow{2}{*}{Matrix Pair}& \multirow{2}{*}{$\tau$}
&\multicolumn{3}{c}{CPF-JDGSVD}
&\multicolumn{4}{c}{JDGSVD}\\
\cmidrule(lr){3-5} \cmidrule(lr){6-9}
&&$I_{out}$&$I_{in}$&$T_{cpu}$
&$I_{out}$&$I_{in}$&$T_{cpu}$&$Relres_t$ \\ \midrule
$(A,B)$&$10^{+3}$  &41&33671&2.07   &658&347376&61.0 &1.96e-8 \\
$(B,A)$&$10^{-3}$  &36&27441&1.67   &2132&1583593&3.45e+2 &1.33e-3 \\
\bottomrule
\end{tabular}
\end{center}
\end{table}

For this problem, both $A$ and $B$ are well conditioned with
$\kappa(A)=4.17\times 10^3$, $\kappa(B)=1.36\times 10^3$ and $\kappa(\bsmallmatrix{A\\B})=26.6$.
The largest and smallest generalized singular values of $(A,B)$ are $4.68\times 10^3$
and $1.21\times 10^{-3}$, respectively.
Table~\ref{table5} displays the results, and Figure~\ref{fig2a} depicts
the accuracy of the converged generalized singular vectors
obtained by CPF-JDGSVD and JDGSVD,
where the accuracy of a converged generalized singular vector triplet
$(\tilde u,\tilde v,\tilde x)$ is measured by
\begin{equation}\label{accuvector}
  \epsilon_{(\tilde u,\tilde v,\tilde x)}=\sqrt{
  \sin^2\angle(\tilde u,u)+\sin^2\angle(\tilde v,v)+\sin^2\angle(\tilde x,x)}
\end{equation}
with the ``exact'' generalized singular vectors $(u,v,x)$ computed by
{\sf gsvd}.

We see from Table~\ref{table5} that
CPF-JDGSVD is much more efficient than JDGSVD, and it
uses less than $10\%$ of outer and inner iterations
and less than $4\%$ of CPU time than the latter.
We have observed that all the generalized singular values computed by the two
algorithms are accurate with the relative errors lying in $[10^{-15},10^{-11}]$.
However, as can be seen from Figure~\ref{fig2},
all the desired
generalized singular vectors computed by CPF-JDGSVD are very accurate
by recalling that we have used the stopping tolerance $tol=10^{-10}$,
and they are two to nearly five orders more accurate than those computed by JDGSVD.
Indeed, we see from Table~\ref{table5} that the relative residual norms of
the converged approximate GSVD components obtained by JDGSVD are a few orders larger than the stopping tolerance $10^{-10}$.
This shows that transforming the GSVD problem into
the generalized eigenvalue problems in \cite{hochstenbach2009jacobi}
is not a general-purpose good choice since a backward stable
algorithm for the generalized eigenvalue problems
cannot produce backward stable approximate GSVD components of
the original GSVD problem, especially when $\kappa(A)$ or $\kappa(B)$
is not small, as proved in \cite{huang2020choices}.
In addition, we have observed that CPF-JDGSVD successively computed
the desired GSVD components of $(A,B)$ one by one correctly
while JDGSVD only succeeded to compute the first six desired GSVD components and
then repeatedly computed the first one after the six ones had converged.
This phenomenon occurs since $B^TB$ is quite ill conditioned with
$\kappa(B^TB)=1.85\times 10^6$ and the right searching subspace
involved in JDGSVD, which should be made $B^TB$-orthogonal to
the converged right generalized singular vectors by solving some
appropriate correction equations, loses $B^TB$-orthogonality to the
converged right generalized singular subspace, so that the information
on the converged GSVD component reappeared and caused repeated
computation of the same GSVD component.

Since the GSVD of $(A,B)$ is equivalent to that of $(B,A)$, we have
also applied CPF-JDGSVD and JDGSVD to $(B,A)$ with
$\tau=10^{-3}$ to compute nine GSVD components of $(B,A)$.
We have found that JDGSVD successfully computes the first eight desired GSVD
components of $(B,A)$. However, the desired generalized singular values of
$(A,B)$ corresponds to the smallest clustered ones of $(B,A)$. It may be
this reason that made
that JDGSVD fail to compute the ninth desired GSVD component when
total $n$ outer iterations have been used and the relative residual norm of
the computed approximate GSVD component could not drop below $10^{-4}$
after $2132$ outer iterations were exhausted.
In contrast, as we see from Table~\ref{table5}, CPF-JDGSVD succeeds
to compute all the desired GSVD components accurately and uses
even fewer outer and inner iterations and less CPU time
than it does when applied to $(A,B)$ with $\tau=10^{3}$.


\begin{exper}\label{exp5}
  We use CPF-JDGSVD and JDGSVD to compute the nine GSVD components
  of the matrix pair $(A,B)=(\mathrm{plddb}^T,B_2^T)$ corresponding to
  the generalized singular values closest to $\tau=70$, where $B_2$ is the
  $n\times(n+2)$ scaled discrete approximation of the second order derivation operator of dimension one:
  \begin{equation*}
    B_2=\begin{bmatrix}-1&2&-1&&
    \\&\ddots&\ddots&\ddots&
    \\&&-1&2&-1\end{bmatrix}
    \in\mathbb{R}^{n\times(n+2)}.
  \end{equation*}
  The desired generalized singular values are the interior ones of $(A,B)$ and highly clustered with each other.
\end{exper}



For this problem, $\kappa(A)=1.23\times10^4$,
$\kappa(B)=1.69\times10^{6}$,
and $\kappa(\bsmallmatrix{A\\B})=1.40\times 10^2$.
Therefore, both $A$ and $B$ are not well conditioned, and it is expected that JDGSVD cannot compute generalized
singular vectors accurately while CPF-JDGSVD works well.
We observe that CPF-JDGSVD uses $96$ outer iterations and
$119874$ inner iterations, about twice of
$44$ outer and $67330$ inner iterations used by JDGSVD.
All the generalized singular values computed by JDGSVD and CPF-JDGSVD are
very accurate with the relative errors lying in $[10^{-15},10^{-13}]$.
Unfortunately, as has been depicted in Figure~\ref{fig2b}, we see that
the generalized singular vectors computed by CPF-JDGSVD are very
accurate and are a few orders more accurate than those computed
by JDGSVD. We have also applied CPF-JDGSVD and JDGSVD to the
matrix pair $(B,A)$ with
$\tau=\frac{1}{70}$ and observed similar accuracy advantage of CPF-JDGSVD
over JDGSVD.


\begin{exper}\label{exp6}
We compute the GSVD components of the other
problems in Table~\ref{table0}:
$(A_a,B_a)\!=(\mathrm{slptsk}^T,B_1)$,
$(A_b,B_b)\!=(\mathrm{rosen10}^T,B_1)$,
$(A_c,B_c)\!=(\mathrm{flower54}^T, B_1)$,
$(A_d,B_d)=(\mathrm{l30}^T,B_1)$,
$(A_e,B_e)=(\mathrm{cq5}^T,B_1)$ and
$(A_f,B_f)=(\mathrm{stat96v5}^T,B_1)$
with the targets $\tau$ being $\tau_a=9 $, $\tau_b=4$, $\tau_c=82$,
$\tau_d=1$, $\tau_e=0.1$ and $\tau_f=4000$, respectively.
\end{exper}

\begin{table}[tbhp]
\caption{Results of applying CPF-JDGSVD to some of the problems in Table \ref{table0}.}\label{table6}
\begin{center}
\begin{tabular}{cccccccccc} \toprule
\multirow{2}{*}{$A$}
&\multicolumn{3}{c}{$\ell=1$}
&\multicolumn{3}{c}{$\ell=5$}
&\multicolumn{3}{c}{$\ell=9$}\\
\cmidrule(lr){2-4} \cmidrule(lr){5-7} \cmidrule(lr){8-10}
&$I_{out}$&$I_{in}$&$T_{cpu}$
&$I_{out}$&$I_{in}$&$T_{cpu}$
&$I_{out}$&$I_{in}$&$T_{cpu}$\\ \midrule
$\mathrm{slptsk}^T$
&7&6589&0.80
&32&23551&3.04
&67&41350&7.02 \\
$\mathrm{rosen10}^T$
&6&2003&0.25
&23&7308&0.95
&41&12373&1.87 \\
$\mathrm{flower54}^T$
&15&39636&14.5
&37&96769&36.5
&66&157342&61.2\\
$\mathrm{l30}^T$
&6&47&0.03
&83&2550&0.66
&96&3641&0.95 \\
$\mathrm{cq5}^T$
&12&35583&10.2
&22&54841&16.4
&30&62921&19.2  \\
$\mathrm{stat96v5}^T$
&12&8350&2.94
&28&16403&5.84
&50&24991&9.69 \\ \bottomrule
\end{tabular}
\end{center}
\end{table}

Notice from \eqref{defB} that the matrices $B$ in the matrix pairs
are rank deficient with $\N(B)={\rm span}\{[1,\dots,1]^T\}$ and
from Table~\ref{table0} that $A_f=\mathrm{stat96v5}^T$ is also rank deficient.
As shown in Table~\ref{table6}, for these problems, CPF-JDGSVD
succeeds to compute all the desired GSVD components, i.e.,
the clustered interior ones of $(A_i,B_i)$, $i=a,b,c,d$,
the clustered smallest ones of $(A_e,B_e)$ and the largest
nontrivial ones of $(A_f,B_f)$.


Finally, we pay special attention to the inner iterations.
The correction equation \eqref{multicorrectionsig} or
\eqref{multicorrectiontau} is {\em symmetric indefinite} and may be
ill conditioned,
which is definitely true when the desired generalized singular values
$\sigma$'s are interior ones or clustered.
When MINRES is used to solve \eqref{multicorrectionsig} or
\eqref{multicorrectiontau}, preconditioning is naturally appealing.
Unfortunately, it is generally hard to
effectively precondition such correction equations.
We have used the MATLAB built-in function {\sf ilu} with $setup.droptol=0.1$
and $0.01$ to compute sparse incomplete LU factorizations of
$A^TA-\theta^2B^TB$ and $A^TA-\tau^2B^TB$ as preconditioners,
and solved the resulting preconditioned {\em nonsymmetric} correction equations
using the BiCGStab algorithm \cite{saad2003}.
We have found that such preconditioners are very often ineffective and,
for many of the test problems, the preconditioned BiCGStab is
inferior to the unpreconditioned MINRES and uses more inner iterations.
Therefore, we do not present the results of using
the preconditioned BiCGStab.

\section{Conclusions}\label{sec:8}

We have proposed a CPF-JDGSVD method
for computing a partial GSVD of the large regular matrix pair $(A,B)$.
In the outer iterations, the method is a standard
Rayleigh--Ritz projection that implicitly solves the mathematically equivalent
generalized eigenvalue problem of $(A^TA,B^TB)$ without
explicitly forming the cross-product matrices, so that it avoids
the possible accuracy loss of the computed GSVD components.
In the inner iterations, the algorithm approximately
solves the correction equations iteratively.
We have established a convergence result on the approximate generalized
singular values and analyzed the inner and outer iterations in some depth.
Based on the results obtained, we have proposed reliable
stopping criteria for the inner iterations. To be more practical,
we have focused on several issues
and have developed a thick-restart CPF-JDGSVD
algorithm with deflation for computing more than one GSVD components of $(A,B)$
corresponding to the generalized singular values closest to
$\tau$.

Numerical experiments have confirmed the efficiency, reliability and accuracy of the
thick-restart CPF-JDGSVD algorithm with deflation
for computing both some interior and extreme GSVD components
of a large regular matrix pair. We have numerically compared
CPF-JDGSVD with JDGSVD and justified the great superiority
of the former to the latter when computing generalized singular vectors
accurately.

\end{document}